\documentclass{amsart}
\usepackage{colonequals}

\usepackage{amsthm,amstext,amscd,amsbsy,amsxtra,latexsym}

\usepackage{cite}

\usepackage[citation-order, msc-links]{amsrefs}

\usepackage{graphicx}
\usepackage{hyperref}

\hypersetup{
  colorlinks   = true, 
  urlcolor     = blue, 
  linkcolor    = blue, 
  citecolor   = red 
}

 \usepackage[usenames,dvipsnames]{color}

\usepackage[capitalise]{cleveref}
\crefname{thm}{Thm.}{}
\crefname{prop}{Prop.}{}
\crefname{lem}{Lem.}{}
\crefname{cor}{Cor.}{}
\crefname{prob}{Problem}{}
\crefname{figure}{Fig.}{}

\usepackage{color}

 \newcommand{\I}{{\rm{i}}}

 \newcommand{\Kb}{{{\bar K}}}

 \newcommand{\Z}{{\mathbb Z}}

 \newcommand{\Q}{{\mathbb Q}}
 \newcommand{\Pp}{{\mathbb P}}
 \newcommand{\CC}{{\mathbb C}}

 \newcommand{\Cc}{{\mathcal C}}
 
 \newcommand{\Ee}{{\mathcal E}}
 \newcommand{\Fc}{{\mathcal F}}
 
  \newcommand{\Hc}{{\mathcal H}}
 \newcommand{\Ll}{{\mathcal L}}
 \newcommand{\KK}{{\mathcal K}}
 \newcommand{\M}{{\mathcal M}}
 
 \newcommand{\OO}{{\mathcal O}}
 \newcommand{\Sc}{{\mathcal S}}

 \newcommand{\Xc}{{\mathcal X}}

 \newcommand{\aut}{{\text{Aut}}}

 \newcommand{\rk}{{\text{rk}}}

 \newtheorem{thm}{Theorem}[section]

 \numberwithin{equation}{section}

\newcommand\iso{\cong}

\def\det{\mbox{det }}

\def\mod{\mbox{ mod }}
\def\deg{\mbox{deg }}

\usepackage{booktabs}
\usepackage{color}

\usepackage{fullpage}

\newcommand{\bigzero}{\mbox{\normalfont\Large\bfseries 0}}

\begin{document}

\title{The splitting fields and  Generators of  Shioda's elliptic surfaces
	 $y^2=x^3 +t^{m} +1$ (I)}
\author{Sajad Salami}
\address{Institute of Mathematics and Statistics, Rio de Janeiro State University, Rio de Janeiro, RJ, Brazil}
\email{sajad.salami@ime.uerj.br}

	\author{Arman Shamsi Zargar}
\address{Department of Mathematics and Applications, University of Mohaghegh Ardabili, Ardabil, Iran}
\email{zargar@uma.ac.ir}
\date{}

\begin{abstract}

The splitting field of an elliptic surface $\mathcal E$ defined over ${\mathbb Q}(t)$ is the smallest subfield  $\mathcal K$ of 
$\mathbb C$ such that ${\mathcal E}({\mathbb C}(t))\cong {\mathcal E}({\mathcal K}(t))$.	
In this paper, we    determine the splitting field ${\mathcal K}_m$
and a set of linearly independent generators for the Mordell--Weil lattice of the Shioda's elliptic surface  with generic fiber given by  
${\mathcal E}_m: y^2=x^3 +t^{m} +1$ over ${\mathbb Q}(t)$  for positive integers $1\leq m\leq 12$.

\bigskip


\bigskip

\noindent	{\bf AMS subjects (2020)}: 14J27, 11G05.

\end{abstract}

\maketitle

\section{Introduction and main results}


In \cite{Shioda1999a}, T.~Shioda stated that for the elliptic surface  with the generic fiber  
defined by the Weierstrass equation
$y^2 =x^3 + t^m +a$, where $m\geq 1$ and $a \in \Q^*$,  the splitting field $\KK |\Q$
is a cyclic extension of a cyclotomic field and probably the minimal vectors of its 
Mordell--Weil lattice give rise  to a subgroup of finite index in the unit group of $\KK$. 
As a concrete example, he determined the splitting field in 
  the case $m=6$ and $a=-1$ without giving any explicit set of generators of its Mordell--Weil lattice.
In \cite{Shioda1999}, T.~Shioda considered certain elliptic surfaces over  $\Q(t)$ with Mordell--Weil lattices 
of type $E_6, E_7$ and $E_8$.
 
In this paper, we  consider the  elliptic surface with generic fiber given by
\begin{equation}
	\label{shi-eq1}
	\Ee_m : y^2=x^3 + t^{m} +1,
\end{equation}
defined over $\Q(t)$.  It is stated by Shioda in \cite{Shioda1992a} and is proved
by H.~Usui  \cite{Usui2000, Usui2001, Usui2006, Usui2008}
  that the  rank of $\Ee_m$ varies from $0$ to $68$ over $\CC(t)$ for  $u \leq m\leq 360 u$ and any
 integer $u\geq 1$. Particularly,   the fibers $\Ee_m$  with  $m=360 u$  have  rank  $68$.
Let $r_m$ denotes the  rank of   Mordell--Weil lattice $\Ll_m=\Ee_m(\CC(t))$. 

	According to the  Table \ref{tabusui}, due to Usui, one can recognize that to determine the linearly independent  generators and the splitting field of the latices $\Ll_m$, it necessary to determine for cases $m=2,3, 4, 5, 6, 9, $ and $12$. The  other cases depend on these and lattice of some K3  surfaces treated by authors in  \cite{Salami2022}, see subsection \ref{Num-res}.

The main goals of the current paper are  to     determine the splitting fields $\KK_m$
 and to provide a set of $r_m$ independent generators of  $\Ee_m(\KK_m(t))$  for   $2\leq m\leq 12$.


	\begin{thm}
	\label{main-a}
	\begin{itemize}
		\item[(i)] For $m=1, 7  $,  the lattices  $\Ll_1\cong\Ll_7$ are   trivial and hence   $r_1=r_7=0$;
		\item[(ii)] For $m=2$,   one has  $ \Ll_2\cong \Ee_2(\KK_2(t))$  with rank $r_2=2$, where	$\KK_2 =\Q(\zeta_3)$ with $\zeta_3 =(\mathrm{i} \sqrt{3} -1)/2$,  defined by a minimal polynomial $g_2(x)=x^2+x+1$.
		A set of   independent generators of $\Ee_2(\KK_2(t))$ contains the following two points: 
$$P_1=(-1, t), \ \  P_2= \left( -\zeta_3 \, ( 1+\frac{ 4}{3} t^2), 
 \I \sqrt{3} \, t ( 1 + \frac{8}{9} t^2) \right).$$
%

	\item[(iii)]	For $m=3$,   one has  $ \Ll_3\cong \Ee_3(\KK_3(t))$   with rank $r_3=4$ 
		where	$\KK_3 =\Q(\zeta_{3}, 2^{\frac{1}{3}} )$, defined by a minimal polynomial $g_3(x)=x^6 - 3x^5 + 5x^3 - 3x + 1.$
		Moreover, a set of independent generators includes the following   four points 
	$$
	\begin{aligned}
		P_1 &=\left( -\zeta_3^2 \, t, 1 \right), &		
		P_3 &=\left( -( t +  2^{\frac{2}{3}}) ,  -(2\zeta_3+1)
		(2^{\frac{1}{3}} t +1) \right),\\
	 P_2 &=\left(-t, 1 \right), &
	 P_4  &=  \left(-(t +   \zeta_3^2\, 2^{\frac{2}{3}}), 
	 (\zeta_3^2-1)\,2^{\frac{2}{3}} t + 2 \zeta_{3}  +1\right). 
	 \end{aligned}
	 $$
	 	\end{itemize}
\end{thm}

We   fix the  $12$-th root of unity $ 	\zeta_{12} = (\mathrm{i}+ \sqrt {3})/2$. 
 Then, for $m=4$, we have: 

	\begin{thm}
	\label{main-c}		
	For $m=4$,   the isomorphic lattices $ \Ll_4\cong \Ee_4(\KK_4(t))$ 
	has rank $r_4=6$, 	and  $\KK_4 $ is  defined by a minimal polynomial $g_4(x)$ of degree 16 given by \ref{g4}, which includes
		$\Q (\zeta_{12},  \alpha_1),\  \text{with}\  \alpha_1=  2^{\frac{1}{4}} 3^{\frac{1}{8}} (\sqrt{3}-1)^{\frac{1}{2}}.$
    Moreover, 	a set of independent generators for $\Ee_4(\KK_4(t))$
includes  points 
 $P_j=(a_j\, t+ b_j, t^2+ c_j\,t +d_j )$ for $j=1,\cdots, 6$, given by \ref{points4}.
\end{thm}


To provide our result in the case $m=5$, 
 we fix the following $30$-th root of unity
	$$\zeta_{30}=\frac{1}{8}\left(\sqrt{3}+\mathrm{I}\right) \left( \left(1-  \sqrt{5}\right) \sqrt{\frac{5+  \sqrt{5}}{2}}+\mathrm{I} \left(\sqrt{5}+1\right)\right).$$


\begin{thm}
	\label{main-d}
			For $m=5$,   one has $r_5=8$ and the isomorphism  $ \Ll_5\cong \Ee_5(\KK_5(t))$  
where $\KK_5 = \Q(\zeta_{30}) \left( (60 v_1)^\frac{1}{30}\right),$  with
$$ v_1=564300+252495\,\sqrt {5}+31\,\sqrt {654205350+292569486\,\sqrt {5}},$$ 
	and  a defining minimal polynomial $g_5(x)$ of degree 120 given in \cite[minpols]{Shioda-Codes}.
	
	%
	Moreover,   $\Ee_5(\KK_5(t))$ is generated by the points
	$$P_j=\left( \frac{t^2 + a_j t +b_j}{u_j^2}, \, \frac{t^3 +c_j t^2 + d_j t +e_j}{u_j^3}\right),$$ 
	for $j=1, \ldots, 8,$ 	where   $u_j$ are given by \ref{u-0-8}
	and the coefficients	$a_j, b_j, c_j, d_j$ and $ e_j$  are given in \cite[points-5]{Shioda-Codes}.
	  \end{thm}

\begin{thm}
	\label{main-e}
	For $m=6$,   one has  $ \Ll_6\cong \Ee_6(\KK_6(t))$  
with rank $r_6=8$ and
$\KK_6(t)=\Q  (\zeta_{12})  \left(2^{\frac{1}{3}}\right),$
with a defining minimal polynomial
\begin{align*}
	g_6(x)	& =x^{12} - 3x^{10} - 8x^9 - 6x^8 + 12x^7 + 47x^6 \notag \\
	& \quad + 78x^5 + 78x^4 + 50x^3 + 21x^2 + 6x + 1\label{g6}
\end{align*}

A set of six independent generators of $\Ee_6(\KK_6(t))$ contains eight   points of the form:
$$P_j=\left( a_j t^2 + b_j t + g_j, \ c_j t^3+ d_j t^2 + e_j t + h_j\right)$$
for $j=1, \ldots, 8$, where  the coefficients are given in Section~\ref{sec-6}.	
\end{thm}

For the case $m=8$, we use Theorem \ref{main-c} to show the following result. 
\begin{thm}
	\label{main-8}
 	For $m=8$,   one has  $ \Ll_8\cong \Ee_8(\KK_8(t))$    with rank $r_8=6$ where  $\KK_8 =\KK_4 $ and the six 
independent generators  of $\Ee_8(\KK_8(t))$ are  of the form
$$Q_j=\left( a_j t^2 + b_j, \ t^4 + c_j t^2+ d_j\right)$$
where  $ j=1, \ldots, 6$ the coefficients of $Q_j$ are same as the points $P_j$ in Theorem \ref{main-c}.
\end{thm}

\begin{thm}
	\label{main-f}
	
	For $m=9$,   one has  $ \Ll_9\cong  \Ee_{9}(\KK_{9}(t))$ 
with rank $r_{9}=10$ and $\KK_{9}$ is defined by a polynomial $g_9(x)$ of degree 54 given in \cite[minpols]{Shioda-Codes}.

Moreover, a set of ten  independent generators of  $\Ee_{10}(\KK_{10}(t))$ are of the form
	$$Q_j=\left(a_{j, 0}+ a_{j, 1} t+ a_{j, 2} t^2+a_{j, 3} t^3,\,  b_{j, 0}+ b_{j, 1} t+ b_{j, 2} t^2+b_{j, 3} t^3 + b_{j, 4}t^4\right), $$
 with coefficients in $\KK_{9}$, and  provided  in \cite[points-9]{Shioda-Codes}.

\end{thm}

For the case $m=10$, we use  part (ii) of  Theorem \ref{main-a} and Theorem \ref{main-d}  to conclude that:

\begin{thm}\label{main-g}
	For $m=10$,   one has  $ \Ll_{10}\cong  \Ee_{10}(\KK_{10}(t))$ 
	with rank $r_{10}=10$ where $\KK_{10}=\KK_{5}$.  
 A set of  ten independent generators for $\Ee_{10}(\KK_{10}(t))$ includes:
	$$
	\begin{aligned}
	Q_1&=(-1, t^5), \ \ Q_2=\left(- \zeta_3 (1+\frac{4}{3} t^{10}),   \I \sqrt{3}\, t^5 (1 + \frac{8}{9} t^{10})\right), \\
	Q_j&=\left(\frac{t^4 + a_j t^2 +b_j}{u_j^2}, \, \frac{t^6 +c_j t^4 + d_j t^2 +e_j}{u_j^3}\right),
	\end{aligned}
	$$
   where for $j=3, \ldots, 10$  the coefficients of $Q_j$ are same as the points  $P_{j-2}$  in Theorem  \ref{main-d}.
\end{thm}

Finally, for the case $m=12$, we proved the following theorem.

\begin{thm} \label{main-h}

   For $m=12$,  one has
     $ \Ll_{12}\cong  \Ee_{12}(\KK_{12}(t))$ 
     has rank   $r_{12}=16$ and 
     $\KK_{12},$  is a number field with a defining minimal  polynomial  of degree 96
     given in \cite[minpols]{Shioda-Codes}.	
     Moreover,  a set of  16 independent generators includes $P_j=(x_j(t), y_j(t) ) $ with
     \begin{align*}
     	x_j(t)& = A_{j,0} t^{4}+A_{j,1} t^{3}+A_{j,2} t^{2}+ A_{j,1} t +A_{j,0}, \\
     	y_j(t) &=B_{j,0} t^{6}+B_{j,1} t^{5}+B_{j,2} t^{4}+B_{j,3} t^{3}+B_{j,2} t^{2}+B_{j,1} t +B_{j,0},
     	\end{align*}
     for $j=1, \ldots, 16$    are as described in Section~\ref{case9}.
\end{thm}

In our computations, we mostly used    {\sf{Maple}} and   {\sf{PARI/GP}}  in {\sf{SageMath}}.

\section{Preliminaries  on elliptic surfaces}
In this section, we provide basic definitions and some of the known results on elliptic surfaces and their Mordell--Weil lattices. 
The interested readers can find more details on the concepts in
\cite{Shioda1990a, Schuett2019, Silverman1994}.
\subsection{Elliptic surfaces and Mordell--Weil lattices}

Let $k \subset \CC$ be a number  field,  $\Cc$ a smooth projective curve
defined over $k$,  $K=k(\Cc)$  and $\Kb=\CC(\Cc)$ the function fields of $\Cc$ over $k$ and $\CC$, respectively.

Given an elliptic curve $\Ee$ over $\Kb$ defined by $ y^2=x^3+ a x + b$ with  $a, b\in \Kb$ and 
discriminant  $\Delta(\Ee) = - 16 (4 a^3+ 27 b) \in \Kb \backslash \CC$,  we denote by $\Ee(\Kb)$ the group of $\Kb$-rational points of $\Ee$
with the identity $\OO \in \Ee(\Kb).$
A generalization of the famous Mordell--Weil theorem over function fields, due to Lang--N\'{e}ron, implies
that $\Ee(\Kb)$  is  a finitely generated abelian group under a mild condition. 
In other words, there is a finite group $\Ee(\Kb)_{tors}$, called the {\it torsion subgroup} of $\Ee$, and a non-negative integer
$r=\rk (\Ee(\Kb))$ called the {\it Mordell--Weil rank} or simply the {\it rank} of $\Ee$ over $\Kb$ such that $\Ee(\Kb) \cong \Ee(\Kb)_{tors} \oplus \Z^r$.

Let  $\pi: \Sc_\Ee \rightarrow \Cc$  be the {\it elliptic surface associated to} $\Ee$ over $\Kb$, i.e. the Kodaira--N\'{e}ron model of 
$\Ee$ over $\Kb$, where  $\Sc_\Ee$ is a smooth complex projective surface
and $\pi$ is a relatively minimal fibration with the generic fiber $\Ee$.
We assume that  $\pi$ is not smooth, i.e, there is at least one singular fiber.
In this case, the set of  $\Kb$-rational points  $\Ee(\Kb) $ can be identified with 
the  global sections of  $\pi: \Sc_\Ee \rightarrow \Cc$, i.e, the set of all morphism $\sigma: \Cc \rightarrow \Sc_\Ee$  
such that    $\pi \circ \sigma =id$ over $\CC$.
For each $P\in \Ee(\Kb)$, we denote by $(P)$ the image curve of the corresponding global section
$\sigma_P: \Cc \rightarrow \Sc_\Ee$. In particular, the trivial point $\OO$ corresponds to $(\OO)$,  the
image curve of the morphism  $\sigma_\OO: \Cc \rightarrow \Sc_\Ee$ of  $\pi$, which is called  the   {\it zero section}.

Let ${\rm NS} (\Sc_\Ee)$ be the {\it N\'{e}ron--Severi group} of $\Sc_\Ee$, i.e, 
the group of divisors modulo algebraic equivalence.
By the non-smoothness assumption on $\pi$,  ${\rm NS} (\Sc_\Ee)$ is an  indefinite integral lattice called the  {\it N\'{e}ron--Severi lattice},
with  respect to the  pairing defined by the intersection pairing $(D\cdot D')$, where $D$ and $D'$ are divisors on $\Sc_\Ee$.
The rank  of N\'{e}ron--Severi lattice is called the {\it Picard number} of $\Sc_\Ee$ and denoted 
by ${\it p}(\Sc_\Ee)$.
Let $T$ be the {\it trivial sublattice} of  ${\rm NS} (\Sc_\Ee)$, which is  generated by the zero section $(\OO)$,  a fiber $F$ and
the irreducible components of all  fibers of $\pi$ not meeting the zero section. Then, 
$$T= \left\langle (\OO), F\right\rangle \oplus \left( \bigoplus_{\nu\in R} T_\nu \right),$$
where $R=\{ \nu \in \Cc | F_\nu=\pi^{-1}(\nu) \ \text{is reducible}\}$, and $T_\nu$ is generated by the irreducible components of $F_\nu$
other than  the identity component. We note that each $T_\nu$ is a root lattice of type $A, D$ and $E$ up to a sign.
The orthogonal complement $L=T^\perp$ of $T$ in ${\rm NS} (\Sc_\Ee)$   is called the {\it essential sublattice}   of  ${\rm NS} (\Sc_\Ee).$

By \cite[Theorem 1.3]{Shioda1990a}, the map $P \mapsto (P)  \mod  T$ induces an isomorphism $\Ee(\Kb) \cong {\rm NS} (\Sc_\Ee)/T$, and hence 
a unique homomorphism $\varphi: \Ee(\Kb) \rightarrow {\rm NS} (\Sc_\Ee) \otimes \Q$ satisfying 
$\varphi(P) \equiv (P) \mod  T\otimes \Q$, $Im(\varphi)$ is perpendicular to $T$, and 
$ker(\varphi)= \Ee(\Kb)_{tors}$. By \cite[Theorem 8.4]{Shioda1990a},  the pairing map 
$\left\langle ,  \right\rangle : \Ee(\Kb) \times \Ee(\Kb) \rightarrow \Q$, which is given by 
$\left\langle P_1, P_2\right\rangle \colonequals  - \left(\psi(P_1)\cdot \psi(P_2) \right)$,
where $P_1, P_2 \in \Ee(\Kb)$ and $(\cdot )$ denotes the intersection number of sections, 
defines the structure of a positive definite lattice on $\Ee(\Kb)/\Ee(\Kb)_{tors}$ which is called the {\it Mordell--Weil lattice} of
$\Ee$ over $\Kb$ or of  $\pi: \Sc_\Ee \rightarrow \Cc$.  

Given any finite covering $f : \Cc' \rightarrow \Cc$ of smooth projective curves of degree $m\geq 2$ over $\CC$ and the elliptic surface   $\pi: \Sc_\Ee \rightarrow \Cc$, we denote by  $\Ee'$ the generic fiber of the  base change $\pi': \Sc_{\Ee'}\colonequals \Sc_\Ee \times \Cc' \rightarrow \Cc'$ and let $\Kb'=\CC(\Cc')$, which is a  finite extension of $\Kb$ of degree $m\geq 2$. Then, 
by \cite[Proposition 8.12]{Shioda1990a}, for any
$P_1, P_2 \in \Ee(\Kb)$, we have 
$\left\langle Q_1 , Q_2\right\rangle' =m \cdot \left\langle P_1\, ,\, P_2\right\rangle,$
where $Q_1 , Q_2 \in \Ee'(\Kb')$  and 
$\left\langle \, ,\,\right\rangle'$ is the height pairing on  $\Ee'(\Kb')$ the set of global sections of $\pi'.$

The elliptic surface $\Sc_\Ee$ 
is called a {\it rational elliptic surface} if  it is birational to $\Pp^2$, the projective space of dimension two.
In this case, $\Cc$ is isomorphic to the projective line $\Pp^1$, hence $\Kb=\CC(t)$   and the generic fiber  
is an elliptic curve  given by the Weierstrass equation  
\begin{equation}
	\label{ell1}
	\Ee: y^2=x^3+ a(t) x + b(t)
\end{equation}
with  $a(t), b(t)\in \CC(t)$ satisfying
$\deg a(t)\leq 4$,  $\deg b(t)\leq 6$ and  $\Delta(\Ee)$ is non-constant.
The  surface $\Sc_\Ee$ is  called an {\it elliptic  $K3$ surface} if   $4<\deg a(t)\leq 8$ and  $6< \deg b(t)\leq 12$ hold;
otherwise, it is called an {\it honestly elliptic  surface}.
The {\it arithmetic genus}  $\Xc (\Sc_\Ee)$ of $\Sc_\Ee$ is the smallest integer $n$ such that $\deg(a(t))\leq 4 n$ and $\deg(b(t))\leq 6 n$.
It is  equal to $1$ and $2$ if $\Sc_\Ee$ is a rational or a $K3$ surface, respectively, 
and     $\Xc (\Sc_\Ee) \geq 2$ for  any  honestly elliptic surface $\Sc_\Ee$.
For example, we have $\Xc (\Sc_\Ee) =60$ for the elliptic surface associated to the Shioda's curve \ref{shi-eq1} with $m=360$, since
$\deg( t^{360}+1)= 6\cdot 60$.

Given $P, P_1, P_2 \in \Ee(\Kb)$, the explicit formulas for height pairing $\left\langle ,  \right\rangle$ are given  as follows:
\begin{equation}
	\label{hf1}
	\begin{cases}
		\left\langle P_1, P_2\right\rangle = \displaystyle \Xc (\Sc_\Ee)+ (P_1\cdot O)+(P_2\cdot O)- (P_1 \cdot P_2) - \sum_{\nu \in R}^{} \text{contr}_{\nu}(P_1, P_2),\\
		\left\langle P, P\right\rangle = \displaystyle 2 \Xc (\Sc_\Ee)+ 2 (P \cdot O) - \sum_{\nu \in R}^{} \text{contr}_{\nu}(P), 
	\end{cases}
\end{equation}
where  $(P_1\cdot P_2)=((P_1), (P_2))$ denotes the intersection
number of the sections $(P_1)$ and $(P_2),$  see  \cite[Theorem 8.6]{Shioda1990a}  for the proof.
The {\it local contribution terms} $\text{contr}_{\nu}(P_1, P_2)$ and 
$ \text{contr}_{\nu}(P)$ are non-negative rational numbers which  depend on the reducible fibers of $\pi$ over $\nu$
and  can be computed according to  Table 8.9 in \cite{Shioda1990a}.
In particular,   both of them are equal to zero when there is no reducible fibers.  

The subgroup $\Ee(\Kb)^\circ$ of $\Ee(\Kb)$, consisting of those sections meeting the identity component of every fiber of $\pi$, 
is a torsion-free subgroup of $\Ee(\Kb) $ of finite index and becomes a positive definite even lattice. 
It is called the {\it narrow Mordell--Weil lattice} of $\Ee$ over $\Kb$, which is   isomorphic  to the opposite lattice $L^-$ of the orthogonal complement  of $T$ in ${\rm NS} (\Sc_\Ee)$. Given  $P_1, P_2 \in \Ee(\Kb)^\circ$, one has the following formulas for the height paring:
\begin{equation}
	\label{hf2}
	\begin{cases}
		\left\langle P_1, P_2\right\rangle  =  \Xc (\Sc_\Ee)+ (P_1 \cdot O)+(P_2 \cdot O)- (P_1 \cdot P_2),\\ \noalign{\medskip}
		\left\langle P, P\right\rangle = 2 \Xc (\Sc_\Ee)+ 2 (P\cdot  O) \geq 2 \Xc (\Sc_\Ee). 
	\end{cases}
\end{equation}
The quotient $\Ee(\Kb)/ \Ee(\Kb)_{tors}$ is contained in the dual lattice $M^*$ of $M=\Ee(\Kb)^\circ$, and the equality holds 
if ${\rm NS} (\Sc_\Ee)$ is an unimodular lattice. In this case, 
$[M^*:M]=\det M=(\det T)/ | \Ee(\Kb)_{tors}|^2$. We cite the reader to \cite[Theorem 9.1]{Shioda1990a} for more details.

The  {\it Shioda--Tate} formula,
 given by  \cite[Crollary 5.3]{Shioda1990a}, 
 provides a relation between the rank of $\Ee(\Kb)$ and the Picard number    $\Sc_\Ee$ as follows:
\begin{equation}
	\label{shi-tate-f}
	\rk(\Ee(\Kb))={\it p}(\Sc_\Ee) -2 - \sum_{\nu \in R}^{}(m_{\nu}-1),
\end{equation}
where $m_{\nu}$ is the number of irreducible components of the reducible  fibers of 
 $\pi$ over $\nu$.
For the  rational elliptic surfaces, since ${\rm NS} (\Sc_\Ee)$ is unimodular of rank  
${\it p}(\Sc_\Ee) = 10$, $\rk(\Ee(\Kb))$ is at most $8$ by the Shioda--Tate formula \ref{shi-tate-f}.
In the case of  elliptic $K3$ surfaces,  the rank $\rk(\Ee(\Kb))$ can  be at most $18$.

The following two theorems completely determine the structure of Mordell--Weil lattice of rational elliptic surfaces.
See \cite[Theorem 10.4]{Shioda1992a} and \cite[Theorem 10.10]{Shioda1992a} for the proofs, respectively.
\begin{thm}
	\label{shi1}
	Given 	a rational elliptic surface 	$\pi: \Sc_\Ee \rightarrow \Pp^1 $  with rank  $\geq 6$, 
	the narrow Mordell--Weil lattice $\Ee(\Kb)^\circ$ is the root lattice $E_8$, $E_7$ or $E_6$, according to whether 
	the morphism	$\pi $ has (i) no reducible fibers, (ii) only one reducible fiber of type $I_2$ or $II$,
	or (iii) one reducible fiber of type $I_3$ or $IV$. Moreover, $\Ee(\Kb)$ is torsion-free and isomorphic to  $E_8$, $E_7^*$ or $E_6^*$ accordingly.
\end{thm}

\begin{thm}
	\label{shi2}
	The Mordell--Weil  group  $\Ee(\Kb)$ of a rational elliptic surface	$\pi: \Sc_\Ee \rightarrow \Pp^1 $  is generated by points (sections) $P$ with $(P\cdot O)=0$, i.e.,
	the points $P$'s with $ \left\langle P, P \right\rangle  \leq 2$.
	If $ \Ee$ is given by \ref{ell1}, then 
	there exist
	at most $240$ points of the form 
	$$P=(a t^2+ b t +g, c t^3+ d t^2 + e t + h),$$
	generating $\Ee(\Kb)$, with $g, h, a, b, c, d, e \in \CC$.
\end{thm}
%
It is known that these points are in one-to-one correspondence with  the  roots of a polynomial
attached to the rational elliptic surfaces, which is described in the next subsection in a general context.
\subsection{Galois representations and fundamental polynomials}
Let $G=\text{Gal}(\CC/k)$, $\Cc$ a smooth projective curve  defined over $k$ and  let $K, \Kb$ be the function field of $\Cc$ over $k$ and $\CC$ respectively.
Let   $\Ee$ be the generic fiber of an elliptic surface $\pi: \Sc_{\Ee} \rightarrow \Cc$ defined over $k$. In a natural way, the Galois group $G$
acts on $\Ee(\Kb)$ and $\Ee(\Kb)^\circ $  coincides with $\Ee(\Kb)^G$ the subgroup of $G$-invariants sections.
By  \cite[proposition 8.13]{Shioda1992a},  the homomorphism    $\varphi: \Ee(\Kb) \rightarrow {\rm NS} (\Sc_\Ee) \otimes \Q$ is
$G$-equivariant and the pairing map $\left\langle , \right\rangle $ on $\Ee(\Kb)$ is stable under action of $G$.
Therefore, we get a Galois representation $$\rho: G =\text{Gal}(\CC/k) \rightarrow \aut\left(  \Ee(\Kb), \left\langle , \right\rangle \right)  $$
where $\aut\left(  \Ee(\Kb), \left\langle , \right\rangle \right)  $ is a finite group.
Let $k \subseteq \KK \subset \CC $ be  the extension of $k$ corresponding to  
$\ker(\rho)$  the kernel of $\rho$ by Galois theory; 
equivalently, $\KK$ is the smallest extension of $k$ such that  $\Ee(\KK(\Cc))=\Ee(\Kb)$.
By the definition $\KK|k$ is a finite Galois extension such that $\text{Gal}(\KK/k) =\text{Im}(\rho)$ the image of $\rho.$ 
Determining  $\text{Im}(\rho)$ is
the main problem concerning  with the Galois representation $\rho$. In particular, one may ask:  How big or how small can $\text{Im}(\rho)$ be?

Let $\nu \in R$, $F_{\nu}=\pi^{-1}(\nu)$ be a singular fiber of the elliptic surface $\pi: \Sc_{\Ee} \rightarrow \Cc$, and
$I$ a $G$-stable finite subset of $\Ee(\Kb)$. The {\it fundamental polynomial} of $\Ee$ over $K$ is defined by
$$\Phi(u)= \prod_{P\in I}^{} (u-\text{sp}_{\nu}(P)),$$
where $\text{sp}_\nu$ is the {\it specialization map} at $\nu$ defined as follows.
For any point $P\in \Ee(\Kb)$, the section $(P)$ intersects the fiber $F_{\nu}$ at a unique smooth point of $F_{\nu}$, say $\text{sp}'_{\nu}(P)$.
The smooth part $F_{\nu}^{\#}$ of the fiber $F_{\nu}$ is an algebraic group over $\CC$ 
isomorphic to the product of ${\mathbb G}_a$ or ${\mathbb G}_m$ by a finite group, and 
the map $\text{sp}'_{\nu}:  \Ee(\Kb) \rightarrow F_{\nu}^{\#}(k)$ is $G$-equivariant homomorphism. Then, $\text{sp}_{\nu}$ is defined to be the projection of 
$\text{sp}'_{\nu}$ to  ${\mathbb G}_a$ or ${\mathbb G}_m$. 

The following result connects the splitting field of $\Phi(U)$ with the extension $\KK$ of $k$ as follows.
\begin{thm}
	\label{phi}
	Keeping the notations as above, we assume that $ \nu \in \Cc(k)$.
	 Then,   $\Phi(u)\in k[u]$ and 
	the splitting field  of the equation $\Phi(u)=0$ is contained in $\KK$, and it is equal to  $\KK$ if 
	   $\text{sp}_{\nu}$ is injective and the set $I$ contains generators of $\Ee(\Kb)$.
\end{thm}
We note that  one may consider the set 
$I_n=\{P\in \Ee(\Kb):  \left\langle P,  P\right\rangle =n\}$ for some $n\geq 1$ in the definition of $\Phi(u)$.
For example, the set $I_2$ is sufficient in the case of rational elliptic surfaces as stated in Theorem~\ref{shi2} and we will see this fact in the next sections.

\subsection{Numerical invariants and the  structure of Mordell--Weil lattice of $\Ee_m$}
\label{Num-res}
First, we fix  some notation and numerical invariants associated to the elliptic surface 
 $\pi_m: \Sc_m \rightarrow \Pp^1$ with generic fiber $\Ee_m$. 
 For any integer $m \geq 1$,  we denote  the N\'{e}ron--Severi group of $\Sc_m$  by $N_m$, and 
its trivial sublattice,    arithmetic genus,   Lefschetz number, respectively, by $T_m$,  $\Xc_m$, and $\lambda_m$ 
 Moreover, we let  $\Ll_m$  be the Mordell--Weil lattice    $\Ee_m  (\CC(t))$ and use  $r_m$,    $d_m$,  $\mu_m$,  and $\tau_m$, to denote
its  rank as a lattice,   its determinant,  its minimal norm,  and   the number of minimal sections  respectively.

\begin{thm} [T.~Shioda \cite{Shioda1991f}]
	\label{invariants}
	For a given integer $m \geq 1$, one has the following statements.
\begin{itemize}	
	\item [(1)] There are  singular fibers of type $II$ over  roots of $t^m   +1=0$. 
	The 	fiber over $t = \infty $ is regular if $m\equiv 0 \pmod 6$, and singular of type $II^*$,
	$IV^*, I_0, IV$ or $II$ according as $m\equiv 1, 2, 3, 4$ or $5 \pmod 6.$
	
		\item [(2)] The trivial sublattice $T_m$ are  isomorphic to the direct sum of  $\left\langle (\OO), F\right\rangle $
		with lattices $0$ for $m\equiv 0, 5 \pmod 6$, and  
			$E_8$, $E_6$, $D_4$, or $A_2$ corresponding to the cases 
		$m\equiv 1, 2, 3, 4 \pmod 6$, respectively.
		The rank of $T_m$ is $2, 10, 8, 6, 4, 2$ for  $m\equiv 0, 1,  \ldots, 5 \pmod 6$, with determinants
		$\det(T_m) =1, 3, 4$  according to the cases  $m\equiv 0, \pm 1  \pmod 6$, 
		$m\equiv 2, 4  \pmod 6$, and   	$m\equiv 3  \pmod 6$.
		\item [(3)] The  arithmetic genus $\Xc_m$ is $m/6$ for  $m\equiv 0  \pmod 6$ and $[m/6]+1$ otherwise.  
			\item [(4)] The Mordell--Weil rank $r_m$ is $2m - 4- \lambda_m$ for  $m\equiv 0  \pmod 6$ and 
			$2m-2 -   \lambda_m$ otherwise. In particular, one has $\Ll_m \iso 0, A_2^*, D_4^*, E_6^*, E_8, E_8$ of ranks
			 $r_m= 0, 2, 4, 6, 8, 8$ for $m=1, 2, \ldots, 6$, respectively.
	\end{itemize}
\end{thm}
\begin{proof}
We cite the reader to  see  Lemma 3.1,     Corollary  3.2,   Propositions 3.3 and 3.4 in \cite{Shioda1991f} for the proof of the above assertions.
\end{proof}

 For any positive numbers $m_1, m_2$,  the  lattice $\Ll_{m}=\Ee_{m}(\CC(t))$ with $m=m_1 m_2$ has a sublattice $\Ee_{m}(\CC(t^{m_2}))$ denoted by 
 $\Ll_{m_1}[m_2]$, whose pairing is $m_2 $ times of the pairing of $\Ll_{m_1}$ and is a primitive sublattive  of  $\Ll_{m}$, i.e.,  $\Ll_{m}/ \Ll_{m_2}$ is torsion-free.
   
  The following result   is proved in a series of papers by  H.~Usui \cite{Usui2000, Usui2001, Usui2006, Usui2008}.
\begin{thm} 
	\label{USUI}
	Keeping the above notations in mind,  for any $m\geq 1$, we have the following:
	\begin{itemize}
		\item [(1)] If $m$ divides $360$, then  the structure of   $\Ll_m$ with associated data are given in Table~\ref{tabusui}.
		\item [(2)] For  general $m$, let $m_1=\gcd(m, 360)$ and $m_2=m/m_1$. Then,
		$$r_m=r_{m_1}, \  \Ll_m =\Ll_{m_1}[m_2], \ d_m=d_{m_1} \cdot m_2^{r_{m_1}},\ \mu_m =m_2 \cdot \mu_{m_1}, \ \tau_m=\tau_{m_1}.$$
	\end{itemize}

\begin{table}[h]
	\centering
	\caption{Data and structure of lattices $\Ll_m$}
	\label{tabusui}
	\begin{tabular}{|c|c|c|c|c|c|} 
		\hline
		$m$ & $r_m$ & $\Ll_m$ & $d_m$ & $\mu_m$ & $\tau_m$ \\  
		\hline\hline
		$1$ & $0$ & $\{0\}$ & $0$ &  & \\ 
		\hline
		$2$ & $2$ & $A_2^*$ & $1/3$ &  $2/3$ &  $6$ \\
		\hline
		$3$ & $4$ & $D_4^*$ & $1/4$ &  $1$ &  $24$ \\
		\hline
		$4$ & $6$ & $E_6^*$ & $1/3$ &  $4/3$ &  $54$ \\
		\hline
		$5$ & $8$ & $E_8^*$ & $1$ &  $2$ &  $240$ \\
		\hline
		$6$ & $8$ & $E_8^*$ & $1$ &  $2$ &  $240$ \\
		\hline
		$8$ & $6$ & $\Ll_4[2]$ & $2^6/3$ &  $8/3$ &  $54$ \\
		\hline
		$9$ & $10$ & $\Ll_9 $ & $3^5/4$ &  $3$ &  $240$ \\
		\hline
		$10$ & $10$ & $\Ll_5[2] \oplus \Ll_2[5] $ & $2^8\cdot 5^2/3$ &  $10/3$ &  $6$ \\
		\hline
		$12$ & $16$ & $\Ll_6[2] + \Ll_4[3] +\widetilde{\Ll_4[3]} + \Hc_3 $ & $2^4 \cdot 3^4$ &  $4$ &  $1848$ \\
		\hline
		$15$ & $12$ & $\Ll_5[3] \oplus \Ll_3[5]  $ & $3^8 \cdot 5^4/4$ &  $5$ &  $24$ \\
		\hline	
		$18$ & $20$ & $\Ll_9[2] +\widetilde{\Ll_9[2]} + \Ll_6[3] $ & $2^{12} \cdot 3^{10}$ &  $6$ &  $674$ \\
		\hline	
		$20$ & $14$ & $\Ll_5[4] \oplus \Ll_4[5]  $ & $2^{16} \cdot 5^{6}/3$ &  $20/3$ &  $54$ \\
		\hline	
		$24$ & $24$ & $\Ll_{12}[2]  +\Hc_4 $ & $2^{20} \cdot 3^{10}$ &  $8$ &  $2040$ \\
		\hline	
		$30$ & $24$ & $\Ll_6[5]  \oplus  \Ll_5[6]  \oplus \widetilde{\Ll_5[6]} $ & $2^{16} \cdot 3^{16} \cdot 5^8$ &  $10$ &  $240$ \\
		\hline
		$36$ & $28$ & $\Ll_{18}[2]  +\Ll_{12}[3] $ & $2^{28} \cdot 3^{22}$ &  $12$ &  $2280$ \\
		\hline	
		$40$ & $14$ & $\Ll_5[8]  \oplus  \Ll_4[10] $ & $2^{30} \cdot 5^{6} /3$ &  $40/3$ &  $54$ \\
		\hline
		$45$ & $18$ & $\Ll_9[5]  \oplus  \Ll_5[9] $ & $3^{21} \cdot 5^{10} /4$ &  $15$ &  $240$ \\
		\hline
		$60$ & $48$ & $\Ll_{12}[5] \oplus \Ll_5[12] \oplus \widetilde{\Ll_5[12]} \oplus \Hc_5 $ & $2^{52} \cdot 3^{36} \cdot 5^{20}$ &  $20$ &  $1848$ \\
		\hline
		$72$ & $36$ & $\Ll_{24}[3] +  \Ll_{18}[4] $ & $2^{56} \cdot 3^{38} $ &  $24$ &  $2472$ \\
		\hline
		$90$ & $36$ & $\Ll_{18}[5]  \oplus  \Ll_5[18]  \oplus \widetilde{\Ll_5[18]} $ & $2^{28} \cdot 3^{42} \cdot 5^{20}$ &  $30$ &  $672$ \\
		\hline
		$120$ & $56$ & $\Ll_{24}[5] \oplus \Ll_5[24] \oplus \widetilde{\Ll_5[24]} \oplus \Hc_5[2] $ & $2^{100} \cdot 3^{44} \cdot 5^{28}$ &  $40$ &  $2040$ \\
		\hline
		$180$ & $60$ & $\Ll_{36}[5] \oplus \Ll_5[36] \oplus \widetilde{\Ll_5[36]} \oplus \Hc_5[3] $ & $2^{76} \cdot 3^{86} \cdot 5^{32}$ &  $60$ &  $2280$ \\
		\hline
		$360$ & $68$ & $\Ll_{72}[5] \oplus \Ll_5[72] \oplus \widetilde{\Ll_5[72]} \oplus \Hc_5[6] $ & $2^{136} \cdot 3^{102} \cdot 5^{40}$ &  $120$ &  $2472$ \\
		\hline
	\end{tabular}
	\centering
\end{table}
	\end{thm}

\begin{proof}  
	 Except for the cases $m=12, 18, 24, 30, 60, 90, 120, 180, 360$,
  the lattice structures $\Ll_m$ are given in \cite[Theorem 2]{Usui2000}.
 The case   $\Ll_9$ is treated in \cite [Corollary 1]{Usui2001}.
 The cases $m=18, 30$ are studied in \cite {Usui2006}. For $m=24, 60$, one can see   \cite[Section 2]{Usui2008}, and
 the cases $m=12, 90, 120, 180, 360$ as well as other data of $\Ll_m$ are treated in \cite[Section 4]{Usui2008}. 
\end{proof}

Let us consider the elliptic $K3$ surfaces   with generic fiber given by 
$ \displaystyle y^2=x^3 + t^n +  1/t^n $  and denote their lattices by $\Fc_n$ for $  1\leq n \leq 6$.
In \cite{Salami2022}, we have investigated the splitting field and independent generators of these $K3$ surfaces.
The following lattices
\begin{align*}
	\Hc_2 &= \left\lbrace  \left( t^2 x(t^3), t^3 y(t^3) \right)\ |\ \left(x(t), y(t)\right)\in \Fc_2 \right\rbrace,  \\
	\Hc_3 & = \left\lbrace  \left( t^2 x(t^2), t^3 y(t^2) \right)\ |\ \left(x(t), y(t)\right)\in \Fc_3 \right\rbrace,  \\
	\Hc_4 & = \left\lbrace  \left( t^4 x(t^3), t^6 y(t^3) \right)\ |\ \left(x(t), y(t)\right)\in \Fc_4 \right\rbrace,  \\
	\Hc_5 & = \left\lbrace  \left( t^{10} x(t^6), t^{15} y(t^6) \right)\ |\ \left(x(t), y(t)\right)\in \Fc_5 \right\rbrace,  \\
	\Hc_6 & = \left\lbrace  \left( t^{2} x(t), t^{3} y(t) \right)\ |\ \left(x(t), y(t)\right)\in \Fc_6 \right\rbrace,
\end{align*}
  are sublattices of $\Ll_{12}, \Ll_{18}, \Ll_{24}$ and $\Ll_{60}$, respectively, and we have
$$\Hc_2 \iso {\mathcal F}_2 [3], \ \Hc_3 \iso {\mathcal F}_3 [2],\ 
\Hc_4 \iso {\mathcal F}_4 [3], \ \Hc_5 \iso {\mathcal F}_5 [6],\  \text{and}\ \Hc_6 \iso \Fc_6.$$
Moreover, for a sublattice   $\Ll$ of $\Ll_{6n}$,   we define  the following lattice
$$\widetilde{\Ll} = \left\lbrace \widetilde{P} =(t^{2n} x(1/t), t^{3n} y(1/t)) \ | \ P=(x(t), y(t)) \in \Ll \subset \Ll_{6n} \right\rbrace,$$
which is isomorphic to $\Ll$ under the automorphism $\iota: P \mapsto \widetilde{P}$ on  $\Ll_{6n}$ such that $\iota (P)=P$.


\section{Proofs of Theorem~\ref{main-a} }
\label{case2-3}
In this section, we prove the Theorem \ref{main-a}  to determine the splitting field and generators of the simplest cases, say   $\Ee_2$ and   $\Ee_3$ as follows.

\subsection{Proof  of parts (i) and (ii)}


It is easy to check that the three obvious points $Q_j=(-\zeta_3^j, t)$ for $0, 1, 2$ belong $\Ee_2(\CC(t))$, 
but not in $\Ee_2(\CC(t)^\circ)$.
 Then, by addition law on $\Ee_2$, we  obtain  the point 
 $$  Q_1-Q_2= \left(- \zeta_3 (1+\frac{4}{3} t^2),  \I \sqrt{3}\, t (1 + \frac{8}{9} t^2)\right) \in \Ee_2(\CC(t)^\circ).  $$
The gram matrix of the points $P_1=Q_1, P_2=Q_1-Q_2$  is equal to 
\begin{equation}\label{mat2}
\M_2=\frac{1}{3}
\begin{pmatrix}
	2 & -1 
	\\ \noalign{\medskip}
	-1 & 2 
\end{pmatrix},
\end{equation}
with determinant $\det (\M_2)=2/3 $ as desired.
Thus, the splitting field  of $\Ee_2$ is equal to $\KK_2=\Q( \zeta_3).$

\subsection{Proof  of part  (iii)}

To simplify the computations, we change $t$ with 
$t-1$, to get the elliptic surface
$$\Ee'_3: y^2=x^3+ t (t^2-3 t +3),$$  which is  isomorphic to $\Ee_3$. The curve $\Ee'_3$ has a singular fiber of type $II$ at $t=0$,
 and the roots of $t^2-3 t +3$. The fiber at $t=\infty$
is of type $I_0^*$. Thus,  
$\Ee_3(\CC(t)) \iso \Ee'_3(\CC(t)) \iso D_4^*, $ as it is mentioned in Theorem~\ref{invariants}.
The generators of $\Ee'_3(\CC(t)) $ are of the form  $P=(a t+ b, ct+ d).$
Substituting the coordinates of $P$ in   the equation of $\Ee'_3,$ we obtain
\begin{equation}
	\label{eq3-1}
a^3+1=0, \ c^2= 3 a^2 b - 3, \ 2 cd = 3 a b^2 +3, \  b^3=d^2.
\end{equation}
Letting  $u=\text{sp}_{t=0}(P)=b/d$ and using   $b^3=d^2$,
 we have $b=u^{-2}$ and $d=u^{-3}.$ Then, using the second and third equations, one can get  
 the fundamental polynomial  $\Phi_3(u)  =27 u^{24}+108 u^{18}-126 u^{12}-8 u^{6}-1.$
Taking $U=u^6$, we write  $\Phi_3(U)  =27 U^{4}+108 U^{3}-126 U^{2}-8 U-1,$ which can be 
decomposed as follows 
$$\Phi_3(U) =27 (U-1)   \left( U + \frac{ (2^{\frac{2}{3}}+1 )^6 }{27} \right) 
\left( U + \frac{ (\zeta_3\, 2^{\frac{2}{3}}+1 )^6 }{27} \right) 
\left(  U +  \frac{ (\zeta_3^2\, 2^{\frac{2}{3}}+1 )^6 }{27}\right). $$
Thus, the splitting field  $\KK_3$ of $\Phi_3(U)$  and hence  $\Phi_3(u)$ contains  $\Q(\zeta_3, 2^{\frac{1}{3}})$. 
Using Maple, one can factor  the fundamental polynomial as follows:
\begin{align*}
\Phi_3(u)	& = (u - 1)(u + 1)(u^2 + u + 1)(u^2 - u + 1)(3u^6 + 3u^4 - 3u^2 + 1)\\
	& \quad \times (3u^6 - 9u^5 + 12u^4 - 9u^3 + 6u^2 - 3u + 1) \\
	&\quad  \times (3u^6 + 9u^5 + 12u^4 + 9u^3 + 6u^2 + 3u + 1)
\end{align*}
	Using Pari/GP and the factors of $\Phi_3(u)$, we find a defining minimal polynomial of degree 6 for  the splitting field $\KK_3$ given by $g_3(x)=x^6 - 3x^5 + 5x^3 - 3x + 1.$

For $U=1$, letting $u_1=-\zeta_3$  and     $a=-\zeta_3^2$,  after changing $t$ with $t+1$, 
one obtains the point
$P_1= \left(-\zeta_3^2 t, 1 \right)$ in $\Ee_3(\KK_3(t))$. 
Taking $u_2=1$, $a=-1$, and using the fact  $c=(3 u^{4}-3)/2u$ and changing $t$ with $t+1$ leads to the second point  	$P_2 =\left(-t, 1 \right)$ in  $\Ee_3(\KK_3(t))$. 
Indeed,  there are  eight points with $a=-1$ and $c=(3 u^{4}-3)/2u$ determined by the roots of the factors
$(u-1)(u+1)(3u^6 + 3u^4 - 3u^2 + 1) $  of $\Phi_3(u)$. 
These points generate a sublattice of index $2$ in $ \Ee_3(\KK_3(t))$ and has
gram matrix equal to the unit matrix of degree four. 
Considering  the following two roots
 $$\displaystyle
 u_3= \frac{1}{3}(2\zeta_3+1)(2^{\frac{2}{3}}   + 1), \ 
 u_4=-\frac{1}{3}(2\zeta_3+1)(\zeta_3\, 2^{\frac{1}{3}}   + 1),$$
 we get the  more   two  points in   $ \Ee_3(\KK_3(t))$ as follows
$$
P_3 =\left( -( t +  2^{\frac{2}{3}}) ,  -(2\zeta_3+1)
(2^{\frac{1}{3}} t +1) \right),\
P_4  = \left(-(t +   \zeta_3^2\, 2^{\frac{2}{3}}), 
  (\zeta_3^2-1)\,2^{\frac{2}{3}} t + 2 \zeta_{3}  +1\right).
$$

One can easily check that the Gram matrix of   $P_1, \cdots, P_4$ is 
\begin{equation}\label{ma4}
	\M_3=\frac{1}{2}
\begin{pmatrix}
	2 & 1 & 1 & 1
	\\ \noalign{\medskip}
	1 & 2 & 0 & 0
	\\ \noalign{\medskip}
	1 & 0 & 2 & 0
	\\ \noalign{\medskip}
	1 & 0 &0 & 2
\end{pmatrix},
\end{equation}
which has determinant $1/4$ as desired.
Thus, we have completed the proof of part  (iii) in Theorem~\ref{main-a}.

\section{Proof of Theorem~\ref{main-c}}
\label{case4}
By Theorem~\ref{invariants},  the Mordell--Weil lattice of $\Ee_4: y^2=x^3+ t^4+1$ over $\CC(t) $ is isomorphic to
$E_6^*$. 
Thus,  a set of generators can be found between $56$ rational points of the form
\[ P=\left( a t +b, \pm t^2 + c t + d\right),\]
with coefficients 
$a, b, c, d \in \KK_4$, where $\KK_4 \subset \CC$ is the splitting field of $\Ee_4$ and  will be determined below.
 
Considering only the positive sign in the $y$-coordinates  and substituting in the equation of $\Ee_4$, leads to the following:
$$2 c -a^{3}=0, \ \  2 c d -3 a \,b^{2}=0, \ \ d^{2}-b^{3}=1, \ \ c^{2}+2 d -3 a^{2} b=0.$$
We consider the specializing map 
$\text{sp}_{\infty}: \Ee_4(\KK_4(t)) \rightarrow \KK_4,$ 
given by 
$ \text{sp}_{\infty} (P)\colonequals a  $ for $P \in \Ee_4(\KK_4(t)).$
 Resolving the above equations leads to a polynomial of degree $27$ in terms of $a$ and hence the fundamental polynomial is equal to 
\begin{align}
	\Phi_4(a)&=  a^3 (a^{24} + 17280 a^{12} - 110592)\notag  \\
	&\quad  = a^3   (a^{8}+24 a^{4}-48) \notag \\
	&\quad \times (a^{16}-24 a^{12}+624 a^{8}+1152 a^{4}+2304).\label{phi4}
\end{align}
The splitting field $\KK_4$ is equal to the compositum of  fields defined by factors of $\Phi_4$.
Using Pari/GP, we find   a defining  minimal polynomial of $\KK_4$ as follows:
\begin{align}
	g_4(x) &= x^{16} - 8x^{15} + 38x^{14} - 120x^{13} + 272x^{12}  - 436x^{11}\notag \\
	&\quad + 472x^{10} - 264x^{9} - 62x^{8} + 216x^{7}  - 128x^{6} - 8x^{5} \notag  \\
	&\quad + 56x^{4} - 48x^{3} + 32x^{2}- 16x + 4.  \label{g4}
\end{align}
Taking $a=0$ in the above  equations leads to $c=d=0$ and $b^3+1=0$. Thus,  one obtains three points of the form $(b,  t^2)$ with $b=-1, \ -\zeta_3, \ -\zeta_3^2$ which span a sublattice of rank $2$.
	We define the points $P_1=(-1, t^2)$,  $P_2=(- \zeta_3, t^2)$ corresponding to $b=-1$ and $-\zeta_3.$
	
Fixing a $24$-th root of uinty $\zeta_{24}=  \sqrt{2}\,(1+\I)(\I-\sqrt{3})/4$, one can decompose  the second factor  of    $\Phi_4(a)$ as follows
$$a^{24} + 17280 a^{12} - 110592=\prod_{j=0}^{11}\left( a- \zeta_{24}^{2j}\, 2^{\frac{1}{4}} 3^{\frac{1}{8}} \sqrt{\sqrt{3}-1}\right) 
\left( a- \zeta_{24}^{2j+1}\, 2^{\frac{1}{4}} 3^{\frac{1}{8}} \sqrt{\sqrt{3}+1}\right).$$
%



Now, to complete the list of generators, we let $\zeta_6=(1 +\I \sqrt{3})/2$ denotes a  $6$-th root of unity and set
$\alpha_1=   2^{\frac{1}{4}} 3^{\frac{1}{8}} \sqrt{\sqrt{3}-1}$.
 Then, we 
 consider the following   points:
  \begin{align}
  	\label{points4}
	P_1& =(-1,t^2), \   \ \ 
	 P_2  =(-\zeta_3,t^2), \\  \notag
	P_3& =\left( \alpha_1 t + (\sqrt{3}-1), \,  t^2+ \frac{\alpha_1^3}{2}\, t + \frac{\sqrt{3} \alpha_1^2}{2} \right), \\  \notag
	P_4& =\left( \zeta_{12}\, \alpha_1 t +  \zeta_3\,(\sqrt{3}-1), \,  t^2+ \frac{\I \, \alpha_1^3}{2}\, t - \frac{\sqrt{3} \alpha_1^2}{2} \right), \\  \notag
	P_5& =\left( \zeta_{6}\, \alpha_1 t +  \zeta_3^2\,(\sqrt{3}-1), \,  t^2- \frac{\alpha_1^3}{2}\, t + \frac{\sqrt{3} \alpha_1^2}{2} \right),\\  \notag
	P_6& =\left( \I\, \alpha_1 t + \sqrt{3}-1, \,  t^2- \frac{\I \, \alpha_1^3}{2}\, t - \frac{\sqrt{3} \alpha_1^2}{2} \right). 
\end{align}

 We note that $\left\langle P_i, P_i\right\rangle=4/3 $ for  $i=1,\cdots,6$. Denoting $x_{ij}=x(P_i)-x(P_j)$ and $ y_{ij}=y(P_i)-y(P_j)$ for
 $ 1\leq  i \neq  j \leq  6$,  the pairing  $\left\langle P_i, P_j\right\rangle  $   can be calculated by 
\begin{align*}  
   \left\langle P_i, P_j\right\rangle  & =\frac{1}{3} - 
   \left\lbrace  \deg \left( \gcd  (x_{ij} , y_{ij})  \right)    
   + \min\left\lbrace 1- \deg(x_{ij}), 2- \deg(y_{ij}))\right\rbrace \right\rbrace.
\end{align*} 
Thus,  the Gram matrix of these six points $P_1, \cdots,P_6$ is equal to the following matrix,
\begin{equation}\label{ma4}
	\M_4=\frac{1}{3}
\begin{pmatrix}
4 & -2 & -2 & 1 & 1 & -2 
\\
-2 & 4 & 1 & -2 & 1 & 1 
\\
-2 & 1 & 4 & -2 & 1 & 1 
\\
1 & -2 & -2 & 4 & -2 & 1 
\\
1 & 1 & 1 & -2 & 4 & -2 
\\
-2 & 1 & 1 & 1 & -2 & 4 
\end{pmatrix},
\end{equation}
whose determinant is equal to $1/3$ as desired for any  lattice isomorphic to $E_6^*$.
 This completes the proof of  Theorem \ref{main-c}. 

\section{Proof of Theorem~\ref{main-d}}
\label{case5}
	Since $\Ee_5 (\CC(t))$ is isomorphic to the root lattice $E_8$, there are $240$ minimal sections corresponding to the points of following form:
	$$P=\left( \frac{t^2+a t + b}{u^2},\frac{t^3+ c t^2 + d  t +3}{u^3} \right).$$
	Putting the coordinates of $Q$ into the equation of $\Ee_5$  and letting $U=u^6$, one  obtain  the following relations among the coefficients:
	
	\begin{align}
		U + 3a - 2c &= 0, & 3a^2 + 3b - 2d - c^2 &= 0, & a^3 + 6ba - 2e - 2dc &= 0, \notag \\
		3b^2 + 3a^2b - 2ec - d^2 &= 0, & 3b^2a - 2ed &= 0, & U + b^3 - e^2 &= 0. \label{eq0-5-1}
	\end{align}

	%
	Using {\sf{Maple}},  we obtain a polynomial $\Phi_5(U)$ of degree $40$ of $U$ which is equal to 
the product of the following polynomials up to a constant, 
	\begin{align*}
		\Phi'_5(U)&=U^{20}-135432000U^{15}+56473225380000U^{10}+2176717249713600000U^5+583200000,\\
		\Phi''_5(U) 	&=U^{20}+66081312000U^{15}-4811512860000U^{10}+1167566400000U^5+583200000.
	\end{align*}
	Changing the variable $ V=U^5/60$,    gives us the following  polynomials 
$$F'(V)=V^4-2257200 V^3+15687007050 V^2+10077394674600 V+45,$$
		$$ F^{''}(V)  =V^4 + 1101355200 V^3 - 1336531350V^2 + 5405400 V + 45,$$
	 corresponding to $\Phi'_5(U)$ and $\Phi''_5(U)$, respectively. The first one
has the following roots
	\begin{equation}
		\label{Zi-1}
		\begin{aligned}
			v_1&=564300+252495\,\sqrt {5}+31\,\sqrt {654205350+292569486\,\sqrt {5}},\\
			& = 564300+ 252495 \sqrt{5} + \delta_1 (76074 \sqrt{5} +170252), \ 
			 \delta_1 = 2^{\frac{1}{2}}\, 3^{\frac{1}{2}} 5^{\frac{1}{4}}  (1+\sqrt{5})^{\frac{1}{2}}/2, \\
			v_2&=564300+252495\,\sqrt {5}-31\,\sqrt {654205350+292569486\,\sqrt {5}}
			= v_1^\sigma, \\
			v_3&=564300-252495\,\sqrt {5}+31\,\sqrt {654205350-292569486\,\sqrt {5}}=v_1^\tau, \\
			v_4&=564300-252495\,\sqrt {5}-31\,\sqrt {654205350-292569486\,\sqrt {5}}=v_1^{\sigma \tau},
		\end{aligned}
	\end{equation}
where   $\sigma$ and $\tau$ denote the  automorphisms which change the sign of  $\sqrt{3}$ and  $\sqrt{5}$, respectively. Similarly, 
 the roots of second one  are as follows:
	\begin{align*}
		v_5&=-275338800+123135255\,\sqrt {5}-31\,\sqrt {157776180962550-70559653172514\,\sqrt {5}},\\
		& =-275338800 + 123135255\, \sqrt{5} + \delta_1 (76074 \sqrt{5} +170252), \\
		 v_6&=-275338800+123135255\,\sqrt {5}+31\,\sqrt {157776180962550-70559653172514\,\sqrt {5}}=  v_5^\sigma,\\ v_7&=-275338800-123135255\,\sqrt {5}+31\,\sqrt {157776180962550+70559653172514\,\sqrt {5}}= v_5^\tau, \\ v_8&=-275338800-123135255\,\sqrt {5}-31\,\sqrt {157776180962550+
			70559653172514\,\sqrt {5}}=v_5^{\sigma \tau}.
	\end{align*}

Moreover, one can check that the ratio of any two roots of $F'(V)$ and $F''(V)$ are units in 
	$\Q(\sqrt{5}, \delta_1) \subset \Q(\zeta_{30})$, see the subsection~4.9 in \cite{Shioda1999} for more details.	
	Thus, all the $240$ roots of fundamental polynomial $\Phi_5(u)$ are of the form 
	$u=\zeta_{30}^\ell (60 v_j)^{1/30}$  for some  $1\leq j \leq 8$ and $0\leq \ell \leq 29$,
	 where we keep the notations given before the statement of Theorem~\ref{main-d}.
	Therefore, one can  conclude that  $\KK_5=\Q(\zeta_{30}) \left( (60 v_1)^\frac{1}{30}\right)$. 
	 Using Pari/GP, we find a defining minimal polynomial for $\KK_5$  of degree 120 as given in \cite[minpols]{Shioda-Codes}.

	Using the $240$ roots $u$ of   $\Phi_5(u)$ and the relations~\ref{eq0-5-1}, one may obtain $240$ minimal sections of the form
	$$P=\left(\frac{t^2+a t + b}{u^2},\frac{t^3+ c t^2 + d  t  +3}{u^3}\right).$$
	By searching among them, we have found  the eight roots $u_1, \ldots, u_8$ given by
	\begin{equation}
		\label{u-0-8}
		\begin{aligned}
			u_1 & =  \zeta_{30}^2(60 v_1^{\tau})^\frac{1}{30}, &  
			u_2 & = (60 v_1)^\frac{1}{30}, &
			u_3 & =  (60 v_2^\tau)^\frac{1}{30}, & 
			u_4 & =  (60 v_2^{\sigma \tau})^\frac{1}{30}, \\
			u_5 & =  \zeta_{30}^{16} (60 v_1^{\tau})^\frac{1}{30},  & 
			u_6 & =  \zeta_{30}^{2} (60 v_1^{\tau})^\frac{1}{30}, & 
			u_7 &=  \zeta_{30}^{16} (60 v_2)^\frac{1}{30},  &  
			u_8 &=  \zeta_{30}^{2} (60 v_2)^\frac{1}{30},\\
		\end{aligned}
	\end{equation}
	such that their corresponding sections $P_1, \ldots, P_8$ 
	form an independent subset of  $\Ee_5(\KK_5(t))$.
	 
	 Since the exact values of $u_i$'s and the coefficients of points $P_i$ for $i=1,\cdots, 8$ are so huge complex numbers,  we have provided them in 
	 \cite[points-5]{Shioda-Codes}.

We have  $\left\langle P_i, P_i\right\rangle=2 $ for  $i=1,\cdots,8$, and  letting $x_{ij}=x(P_i)-x(P_j)$ and $ y_{ij}=y(P_i)-y(P_j)$ for
$ 1\leq  i \neq  j \leq 8$,  the pairing  $\left\langle P_i, P_j\right\rangle  $   can be calculated by 
\begin{align*}  
	\left\langle P_i, P_j\right\rangle  & =1 - 
	\left\lbrace  \deg \left( \gcd  (x_{ij} , y_{ij})  \right)    
	+ \min\left\lbrace 2- \deg(x_{ij}), 3- \deg(y_{ij}))\right\rbrace \right\rbrace.
\end{align*} 
Hence,	the   gram matrix of these points is the following unimodular matrix,
\begin{equation}
	\M_5 =
	\begin{pmatrix}
		\label{m2}
	2 & 1 & 1 & 1 & 0 & 0 & 0 & 1 
	\\ \noalign{\medskip}
	1 & 2 & 0 & 1 & 1 & 1 & 0 & 1 
	\\ \noalign{\medskip}
	1 & 0 & 2 & 1 & 0 & -1 & 1 & 0 
	\\ \noalign{\medskip}
	1 & 1 & 1 & 2 & 1 & 0 & 1 & 0 
	\\ \noalign{\medskip}
	0 & 1 & 0 & 1 & 2 & 0 & 1 & 0 
	\\ \noalign{\medskip}
	0 & 1 & -1 & 0 & 0 & 2 & -1 & 1 
	\\ \noalign{\medskip}
	0 & 0 & 1 & 1 & 1 & -1 & 2 & -1 
	\\ \noalign{\medskip}
	1 & 1 & 0 & 0 & 0 & 1 & -1 & 2 	
	\end{pmatrix},
\end{equation}
as desired. Therefore, we have completed the proof of Theorem~\ref{main-d}.

\section{Proof of Theorem~\ref{main-e}}\label{sec-6}
\label{case6}
Since  $\Ee_6(\CC(t))$ is isomorphic to $E_8$ and its
 the $240$ minimal roots    correspond to  $ 240$ points in $  \Ee_6(\CC(t))$ of the form:
\begin{equation}
		\label{eeqq-0}
P=\left(at^2+bt+g, ct^3+dt^2+et+h\right),
\end{equation}
for suitable constants $a, b, c, d, e, h, g \in \CC$. 

In order to ease the computations, we consider the following elliptic surface
$\Ee_6^-:   y^2=x^3+ t^6-1. $
Indeed, the map $\phi: \Ee_6^- (\CC(t))  \rightarrow \Ee_6 (\CC(t))  $ defined by  $(x,y,t)\mapsto (-x,y/\I, t/\zeta_6)$ 
is an isomorphism and  its inverse map is 
 $\phi^{-1}: \Ee_6 (\CC(t)) \rightarrow \Ee_6^- (\CC(t)) $ defined by  $(x,y,t)\mapsto (-x, \I y,  \zeta_6 t)$.

Changing   the variable $t'=t-1$, we assume that any  rational point on $\Ee_6^-$ are of the form:
\begin{equation}
	\label{eeqq-1}
	P'=\left(a' {t'}^2+b'  {t'} + g', c' {t'}^3 + d' {t'}^2+e' {t'} +h'\right).
\end{equation}
Then,  substituting the coordinates of $P'$    into  equation  of  $\Ee_6^- $ 
 and letting $g' = u^{-2}, h' = u^{-3}$, we get  the following six relations:
\begin{equation}
	\label{eeqq6}
	\begin{cases}
		a'^3-c'^2+1=0,  \\
		3 a'^{2} b' -2 c' d' -6 =0, \\
		3 a' \,b'^{2} u^{2}-2 c' e' \,u^{2}-d'^{2} u^{2}+3 a'^{2}+15 u^{2}=0, \\
		b'^{3} u^{3}-2 d' e' \,u^{3}+6 a' b' u -20 u^{3}-2 c'=0, \\
		e'^{2} u^{4}-3 b'^{2} u^{2}-15 u^{4}+2 d' u -3 a'=0, \\
		6 u^{4}+2 e' u' -3 b'= 0,
\end{cases}\end{equation}
which give rise to a degree $240$ fundamental polynomial $\Phi_6 (u)$ with integral coefficients. 
In \cite[Thm. 1]{Shioda1999a}, T.~Shioda   proved that   $\Phi_6 (u)$ can be decomposed over  $\Q(\zeta_{12})\left(2^{\frac{1}{3}} \right)$ in to linear factors.
Thus, the splitting field of     elliptic surfaces $\Ee_6^-$ and hence $\Ee_6$ is 
$\KK_6=\Q(\zeta_{12})\left(2^{\frac{1}{3}} \right)$.
Using Pari/GP, we find a defining minimal polynomial for $\KK_6$ as follows:

\begin{align}
g_6(x)	& =x^{12} - 3x^{10} - 8x^9 - 6x^8 + 12x^7 + 47x^6 \nonumber \\
	& \quad + 78x^5 + 78x^4 + 50x^3 + 21x^2 + 6x + 1\label{g6}
\end{align}

In fact, by taking $U=u^{12}$, one gets a polynomial  $\Phi_6 (U)$ of degree 20 which  decomposes into six irreducible factors in $\Z[U]$, namely,	
$\Phi_6 (U) = \prod_{i=1}^{6} \Phi_{6, i}(U)$
where
\begin{align*}
	\Phi_{6, 1}(U) &= U-1, \ \ \ \
	\Phi_{6, 2}(U)  = 4U+1,\\
	\Phi_{6, 3}(U) &= 729U^3-17739U^2-189U-1,\\
	\Phi_{6, 4}(U) &= 46656U^3+3888U^2+1728108U+1, \\
	\Phi_{6, 5}(U) &= 2176782336U^6+49703196672U^5+4643867821824U^4-606248250624U^3\\
	&\quad +273143664U^2-43848U+1,\\
	\Phi_{6, 6}(U) &= 2176782336U^6+766590179328U^5+870778439424U^4+333394631424U^3\\
	&\quad -2638190736U^2+99673848U+1. \end{align*}
In \cite[Thm. 1]{Shioda1999a}, it is showed that
the roots of  polynomials $\Phi_{6, i} (U)$ for $i=1,\ldots,6$ are all $12$-th power of some elements in 
$\KK_6 $. Thus, any root of the fundamental polynomial  $\Phi (u)$ is of the form $\zeta_{12}^{i} U_\ell^{\frac{1}{12}}$ for $i=0,\cdots, 11$ and
$\ell=1,\cdots 20$, where $U_j$ is a root of   $\Phi_6 (U)$.
Taking $U_1=1$, $U_2=-1/4$, and considering  two roots $U_3, U_4$ of $\Phi_{6, 3}(U)$,  one  root $U_5$ of   $\Phi_{6, 4}(U)$,  two roots $U_6$ and $U_7$ of  $\Phi_{6, 5}(U)$, and  finally one   root  $U_8$  of $\Phi_{6, 6}(U)$, we obtained the followings:
\begin{align*}
u_1&=1, \ \ \ \ 
u_2 =
\left( \frac{-1}{4}\right)^{\frac{1}{12}}, \\
 u_3&=\left(\frac{46 2^{\frac{2}{3}}+58  2^{\frac{1}{3}} +73}{9} \right)^{\frac{1}{12}}=\frac{1+ 2^{\frac{1}{3}}}{\sqrt{3}},\\
 u_4&=\left( \frac{46 {\zeta^1_3} 2^{\frac{2}{3}}+58 \zeta_3  2^{\frac{1}{3}} +73}{36}\right)^{\frac{1}{12}}=\frac{1+\zeta_3  2^{\frac{1}{3}}}{\sqrt{3}},\\
 u_5&= \left( \frac{ 80  \zeta_3  \,2^{\frac{2}{3}}-100  {\zeta^2_3}  2^{\frac{1}{3}} -1}{36}\right)^{\frac{1}{12}}=\frac{(\I -1)( 2-\zeta_3^2  2^{\frac{1}{3}})}{2 \sqrt{3}}\\
u_6&= \left( \frac{\left(337+194 \sqrt{3}\right) 2^{\frac{2}{3}}-\left(314+182 \sqrt{3}\right)  2^{\frac{1}{3}} -(137+80 \sqrt{3})}{36}\right)^{\frac{1}{12}}=
\frac{(1+\sqrt{3})  2^{\frac{1}{3}} -2}{2 \sqrt{3}}\\
u_7&= \left( \frac{\left(337-194 \sqrt{3}\right) 2^{\frac{2}{3}}-\left(314-182 \sqrt{3}\right)  2^{\frac{1}{3}} -(137-80 \sqrt{3})}{36}\right)^{\frac{1}{12}}=\frac{(1-\sqrt{3} )  2^{\frac{1}{3}} -2}{2 \sqrt{3}},\\
u_8	&= \left( -\frac{\left(1327+766 \sqrt{3}\right) 2^{\frac{2}{3}}+2 \left(833+481 \sqrt{3}\right)  2^{\frac{1}{3}} +2113+1220 \sqrt{3} }{36}\right)^{\frac{1}{12}} 
=\frac{(\I+1)(1+\sqrt{3} +  2^{\frac{1}{3}})}{2 \sqrt{3}}
\end{align*}

Letting  $g'_j = u_j^{-2},  h'_j = u_j^{-3}$  
and solving the equations  \ref{eeqq6} leads to  eight points $P'_j$  of the form \ref{eeqq-1} in $\Ee_6^- (\KK_6(t'))$    for $j=1,\cdots 8$.
Then, changing $t'$ with $t-1$ and using the map $\phi$, one can get the following eights points 
generating $\Ee_6 (\KK_6(t))$:
\begin{align*}
P_1& =\left(-1, -t^3 \right), \ \ \ 
 P_2=\left(- 2^{\frac{1}{3}} t , -t^3 +1\right), \\
P_3& =\left( -(\zeta_3^2 t^2+   2^{\frac{2}{3}}),  -\I \, \sqrt{3}\,( \zeta^{2}_{3}\, 2^{\frac{1}{3}} t^2   +1  )\right),\\ 
P_4&=\left(-(t^2+ 2^{\frac{2}{3}}) , \I \,  \sqrt{3}\,( 2^{\frac{1}{3}} t^2 +   1) \right),\\
P_6&=\left( 2 \,\zeta_3^2 t^2-\sqrt{3}\, \zeta_{12}\, 2^{\frac{1}{3}} t - 2^{\frac{2}{3}},
-3 t^3-2\, \sqrt{3}\, \zeta^{-1}_{12}\, 2^{\frac{1}{3}} t^2+3 \,  \zeta_3\, 2^{\frac{2}{3}} t+\I \, \sqrt{3}\right),\\
P_7&=\left(2 \,\zeta_3^2 t^2+\sqrt{3}\, \zeta_{12}\, 2^{\frac{1}{3}} t- 2^{\frac{2}{3}}, 
-3t^3+2\, \sqrt{3}\, \zeta^{-1}_{12}\, 2^{\frac{1}{3}} t^2+3 \,  \zeta_3\, 2^{\frac{2}{3}} t- \I \, \sqrt{3} \right),\\
\end{align*}
and $P_j=\left(a_j t^2+b_j t+g_j, c_j t^3+d_j t^2+e_j t+h_j\right)$ for $j=5,$ and 8 have
 coefficients  as follows:

\begin{align*}
	a_5 &= -\left(\zeta_3\,  2^{\frac{2}{3}} +    2\, \zeta_3^2 \,2^{\frac{1}{3}} +2\right), &
	a_8 & = -\zeta^2_3\,\left( 2^{\frac{2}{3}} +    2\,   \,2^{\frac{1}{3}} +2\right),  \\
	b_5 &= 2\, 2^{\frac{2}{3}} + 3\,  \zeta_3\,  2^{\frac{1}{3}} + 4\, \zeta_3^2,  & 
	b_8 & =  - \left(2 \, 2^{\frac{2}{3}} + 3 \,2^{\frac{1}{3}} +4\right), \\
	g_5 &= -\left(\zeta_3^2\,  2^{\frac{2}{3}} +    2\, 2^{\frac{1}{3}} +2\, \zeta_3\right), & 
	g_8 & = -\zeta_3 \,\left( 2^{\frac{2}{3}} +    2\,   \,2^{\frac{1}{3}} +2\right), \\
	c_5 &=-2 \sqrt{3}\, \left(\zeta_{12}\,  2^{\frac{2}{3}} - \zeta_{12}\, 2^{\frac{1}{3}} -\frac{3}{2} \I \right),
	& c_8 & = -2 \, \I \, \sqrt{3}\, \left( 2^{\frac{2}{3}} +  2^{\frac{1}{3}} +\frac{3}{2} \right), \\
	d_5 &= - \sqrt{3}\, \left(5 \I \,  2^{\frac{2}{3}} - 6 \zeta_{12}\, 2^{\frac{1}{3}} +8\, \zeta^{-1}_{12} \right), & 
	d_8 & = - \sqrt{3}\,\zeta_{12}  \left(5 \,  2^{\frac{2}{3}} + 6 \, 2^{\frac{1}{3}} +8 \right), \\
	e_5 &= - \sqrt{3}\, \left(5 \, \zeta^{-1}_{12} 2^{\frac{2}{3}} - 6\,\I 2^{\frac{1}{3}} -8\, \zeta_{12} \right), & 
	e_8 & =   \sqrt{3}\,\zeta^{-1}_{12}  \left(5 \,  2^{\frac{2}{3}} + 6 \, 2^{\frac{1}{3}} +8 \right), \\
	h_5 & =   2\, \sqrt{3}\, \left(\zeta_{12} 2^{\frac{2}{3}} +\zeta^{-1}_{12} 2^{\frac{1}{3}} -\frac{3}{2} \I  \right),  & 
	h_8 & = 2 \, \I \, \sqrt{3}\, \left( 2^{\frac{2}{3}} +  2^{\frac{1}{3}} +\frac{3}{2} \right),
\end{align*}

It is easy to see that the set   $\{u_1, u_2, \cdots u_8\}$ is an independent set of elements in $\KK_6$, and hence   ${P_1, \cdots, P_8}$ is an independent set in $\Ee_6 (\KK_6(t))$. In a same way as in the previous section,  one can see that the  
Gram matrix  is   the following unimodular matrix:
\begin{equation}\label{mat6}
	\M_{6}=
	\begin{pmatrix} 
		2 & 0 & 0 & 0 & 1 & 0 & 0 & 1 
		\\ \noalign{\medskip}
		0 & 2 & 1 & 0 & 1 & 1 & 0 & 0 
		\\ \noalign{\medskip}
		0 & 1 & 2 & 1 & 1 & 0 & 0 & 1 
		\\ \noalign{\medskip}
		0 & 0 & 1 & 2 & 1 & 0 & 0 & 1 
		\\ \noalign{\medskip}
		1 & 1 & 1 & 1 & 2 & 0 & 0 & 1 
		\\ \noalign{\medskip}
		0 & 1 & 0 & 0 & 0 & 2 & 0 & 0 
		\\ \noalign{\medskip}
		0 & 0 & 0 & 0 & 0 & 0 & 2 & 1 
		\\ \noalign{\medskip}
		1 & 0 & 1 & 1 & 1 & 0 & 1 & 2 
	\end{pmatrix}. 
\end{equation}
Therefore, we have completed the proof of Theorem~\ref{main-e}. We refer the reader to  
\cite[check-6, points-6]{Shioda-Codes}, for details of  computations in this section.

\section{Proof of Theorem~\ref{main-8}}
\label{case8}
Thanks to Table~\ref{tabusui}, the  lattice $\Ll_8$ is equal to $\Ll_4[2]$, i.e.,   
it is  isomorphic to  
$\Ee_4(\CC(t^2)) \iso \Ee_4(\KK_4 (t^2))$.
Its rank   is $r_8=6$ where  $\KK_8 =\KK_4$ and the six 
independent generators $Q_j$ are obtained  by changing $t$ with $t^2$ in the coordinates of  $P_1,\cdots,P_6$ given by \ref{points4}, i.e., 
$$Q_j   =\left( a_j t^2 + b_j, \ t^4 + c_j t^2+ d_j\right), \ \ j=1, \ldots, 6.$$
Using the base change property of the height pairing, one can get the Gram matrix of these points as
\begin{equation}\label{ma4}
	\M_8=\frac{2}{3}
	\begin{pmatrix}
		4 & -2 & -2 & -2 & -2 & 1 
		\\ \noalign{\medskip}
		-2 & 4 & 1 & 1 & 1 & 1 
		\\ \noalign{\medskip}	
		-2 & 1 & 4 & 1 & 1 & -2 
		\\ \noalign{\medskip}	
		-2 & 1 & 1 & 4 & 1 & -2 
		\\ \noalign{\medskip}
		-2 & 1 & 1 & 1 & 4 & 1 
		\\ \noalign{\medskip}
		1 & 1 & -2 & -2 & 1 & 4 
	\end{pmatrix},
\end{equation}
with determinant  $\det(\M_8)=2^6/3$ as in  Table \ref{tabusui}.
Therefore, we have finished the  proof   of Theorem~\ref{main-8}. 
\section{Proof of Theorem \ref{main-f}}
\label{case9}
In order to prove Theorem \ref{main-f} on $\Ee_9$, we change $t$ with $t-1$  and consider the isomorphic surface
$$\Ee'_9: y^2 = x^3+ (t-1)^9+1.$$
By a same argument as \cite[Lemma 10.9]{Shioda1990a}, the minimal sections of $\Ee'_9(t)$ are of the form $Q=\left( x(t), y(t)\right),$ with
	\begin{align}
		x(t)&= a_0 + a_1 (t-1)+a_2 (t-1)^2+ a_3 (t-1)^3,  \notag \\ 
		y(t)&= b_0 + b_1 (t-1)+b_2 (t-1)^2+ b_3 (t-1)^3+b_4 (t-1)^4+ b_5 (t-1)^5. 	\label{eq9-0}
\end{align} 
Substituting \ref{eq9-1} in the defining equation of $\Ee'_9$ and looking at the coefficients of $(t-1)^j$ for $j=0, \ldots, 5$ we have $b_5=0$ and the following:
\begin{equation}
	\label{eq9-1}
	\begin{cases}
		a_0^3=b_0^2, \ \ \   
		3 a_0^2 a_1 + 9 = 2 b_0 b_1, \\ \noalign{\medskip}
		3(a_0^2 a_2 + a_0 a_1^2) -36 =  b1^2 + 2 b_0 b_1 , \\   \noalign{\medskip}
		3 a_0^2 a_3+ 6 a_0 a_1 a_2 + a_1^3+84 =  2(b_1 b_2+ b_0 b_3), \\ \noalign{\medskip}
		6 a_0 a_1 a_3+ 3 a_0 a_2^2 + 3 a_1^2 a_2 -126 = b_2^2 + 2 (b_0 b_4+ b_1 b_3), \\ \noalign{\medskip}
		6 a_0 a_2 a_3+ 3  a_1^2a_3 + 3 a_1^2 a_2 -126 = 2 (b_1 b_4+ b_2 b_3), \\ \noalign{\medskip}
		3 a_0 a_3^2+ 6  a_1 a_2 a_3 + a_2^3-84 =  2b_2 b_4+b_3^2, \\ \noalign{\medskip}
		3(a_1 a_2^2 +  a_2^2 a_3) +36 =   2 b_3 b_4, \\ \noalign{\medskip}
		3 a_2 a_3^2 - 9 = b_4^2, \ \  \
		a_3^3 +1 =0.
\end{cases}\end{equation}

Let $u=\text{sp}_1(P)= a_0/b_0$. By the first equation of \ref{eq9-1}, we have $a_0=u^{-3}$ and $b_0=u^{-2}$.
Using {\sf{Maple}}, one can see that the fundamental polynomial determined by  above equations with respect $u$ is  a polynomial   $\Phi(u)$ of degree $240$, see \cite[check-9]{Shioda-Codes}
The splitting field $\KK_9$ is determined by the roots of this polynomial. Using Pari/GP, we find a defining minimal polynomial $g_9(x)$ of degree 54 for $\KK_9$, see \cite[minpols]{Shioda-Codes}

For a point $P$ determined by   roots of $\Phi(u)$, one can get  more five points in $\Ee_9(\KK_9(t))$ as
$P, \zeta_3 P, \  \zeta_3^2 P, \ -P, \  -\zeta_3 P, $ and $  -\zeta_3^2 P.$
Here,  we have  $\zeta_3 P = (\zeta_3 x (t), y(t))$ and  $\zeta_3^2 P = (\zeta_3^2 x (t), y(t))$.
Since $\zeta_3^2 P= -P  -\zeta_3 P$, these  points are described by $P$ and $\zeta_3 P$. 
Hence, we   consider only the $80$ points determined  by taking $a_3=-1 $  and $\zeta_3$.

Since $\Ee_3(\KK_3(t))$ is of rank four   and   $\Ll_3[3] \subset \Ll_9$, so  changing $t$ with $t^3$ in  their coordinates 
gives us the following  points  in $\Ee_9(\KK_9(t))$
\begin{align}
	Q_1 &=\left( -\zeta_3^2 \, t^3, 1 \right), &		
	Q_3 &=\left( -( t^3 +  2^{\frac{2}{3}}) ,  -(2\zeta_3+1)
	(2^{\frac{1}{3}} t^3 +1) \right), \notag \\
	Q_2 &=\left(-t^3, 1 \right), &
	Q_4  &=  \left(-(t^3 +   \zeta_3^2\, 2^{\frac{2}{3}}), 
	(\zeta_3^2-1)\,2^{\frac{2}{3}} t^3 + 2 \zeta_{3}  +1\right).  \label{p1-4a}
\end{align}

In order to find  more six generators,  we write the fundamental polynomial   as 
$\Phi(u)= \Phi^{'}(u) \cdot \Phi^{''}(u)$,  where $\Phi^{''}(u)$ is a polynomial of degree $160$ and 
$\Phi^{'}(u) = \Phi_{0}(u)   \cdot \Phi_{1}(u)  \cdot \Phi_{2}(u)  \cdot \Phi_{3}(u) \cdot \Phi_{4}(u) $  with

	\begin{align*}
		\Phi_{0}(u) & = u^2 -1,\  \  \ 
		\Phi_{1}(u)   =  3 u^{6}+3 u^{4}-3 u^{2}+1, \\
		\Phi_{2}(u) & = 243 u^{18}+729 u^{16}+972 u^{14}+648 u^{12}+810 u^{10}+486 u^{8}-72 u^{6}+36 u^{4}-9 u^{2}+1,\\
		\Phi_{3}(u) & = 19683 u^{27}+177147 u^{26}+767637 u^{25}+2145447 u^{24}+4369626 u^{23}+6889050 u^{22}\\& +8581788 u^{21}+8345592 u^{20}
		  +6016437 u^{19}+2740311 u^{18}+216513 u^{17}-614547 u^{16}\\
		  &-198288 u^{15}+349920 u^{14}+419904 u^{13}+180792 u^{12}
		 -8019 u^{11}-43983 u^{10}\\
		 & -14661 u^{9}+5589 u^{8}+7614 u^{7}+4158 u^{6}+1620 u^{5}+432 u^{4}+27 u^{3}-27 u^{2}-9 u -1,\\
		\Phi_{4}(u) & = 19683 u^{27}-177147 u^{26}+767637 u^{25}-2145447 u^{24}+4369626 u^{23}-6889050 u^{22}\\ 
		&+8581788 u^{21}-8345592 u^{20}
		  +6016437 u^{19}-2740311 u^{18}+216513 u^{17}+614547 u^{16}\\ 
		  &-198288 u^{15}-349920 u^{14}+419904 u^{13}-180792 u^{12}
		 -8019 u^{11}+43983 u^{10}\\
		 &-14661 u^{9}-5589 u^{8}+7614 u^{7}-4158 u^{6}+1620 u^{5}-432 u^{4}+27 u^{3}+27 u^{2}-9 u +1.
\end{align*} 

We notice that the point $Q_1$ can be obtained by letting $u= \zeta_3^2$, a root of $u^2+u+1$ which divides the  factor $\Phi_2(u)$ of fundamental polynomial,  then solving  the equations \ref{eq9-1}   and changing  $t$ with $t+1$. Similarly,  $u=1$  give us $Q_2$ in   above list. The   factor of $\Phi_{1}(u)$    can be factored over $ \Q(\zeta_{12}, 2^{\frac{1}{3}})$ as
$$\Phi_{1}(u)=3 (u-v_{11})(u+v_{11})(u-v_{12})(u+v_{12})(u-v_{13})(u+v_{13}), \ \text{where}$$
$$v_{11}=\frac{\I  \sqrt{3}}{3}\, \left(2^{\frac{1}{3}}+1\right), \, 
v_{12}= \frac{\zeta_{12} \sqrt{3}}{3}( 2^{\frac{1}{3}} + \zeta_{3}^2), \,
v_{13}= \frac{\I\sqrt{3}}{3}\left(  \zeta_{3}^2 2^{\frac{1}{3}} +1\right).
$$
Taking $u_3=v_{11}$ and $u_4=v_{12}$ give us the points $Q_3$ and $Q_4$ as above.

In what follows, we will use the roots of  other factors of $\Phi^{'}(u)$ to complete the list of ten independent generators of  $\Ee_9(\KK_9(t))$, where  $\KK_9$ is an extension of $ \Q(\zeta_{12}, 2^{\frac{1}{3}})$ that all factors of
   $\Phi(u)$ decompose completely.
Using Maple, see  \cite[deg18]{Shioda-Codes}, we find out that the factor $\Phi_{2}(u)$ can be decomposed as
{\small 
	\begin{align*}
		\Phi_{2}(u) &= 243 \cdot \prod_{i=1}^{3}\prod_{j=0}^{1} (u-v_{ij}) \times \prod_{i=1}^{3}\prod_{j=0}^{1} (u-v'_{ij}) \times 
		\prod_{i=1}^{3}\prod_{j=0}^{1} (u-v^{''}_{ij}) ,\ \text{where} 
\end{align*}
}
$$
\begin{aligned}
	v_{ij}&=  (-1)^j \frac{ \I \sqrt{3}}{3}\left(  \zeta_{3}^i \left(  2^{\frac{1}{3}} ( 2^{\frac{1}{3}} -1) \right)^{\frac{1}{3}} + 1  \right),  \\
	v^{'}_{ij}&=\frac{\sqrt{3}}{3}\left(\zeta_3^i \left( (-1)^j \zeta_{12} 2^{\frac{1}{3}} ( 2^{\frac{1}{3}} +\zeta_{12}^{10}) \right)^{\frac{1}{3}} + \I \, (-1)^j  \right),\\
		v^{''}_{ij}&=\frac{\sqrt{3}}{3}\left(\zeta_3^i \left( (-1)^j \zeta_{12} 2^{\frac{1}{3}} (\zeta_{12}^{10} 2^{\frac{1}{3}} +1) \right)^{\frac{1}{3}} +\I \, (-1)^{j+1} \right).
\end{aligned}
$$

Now, we consider the following roots of $\Phi_{2}(u)$:

\begin{equation}
	\label{u5-9}
\begin{aligned}
u_5:=v_{00} & =   \frac{ \I \sqrt{3}}{3}\left( \left(  2^{\frac{1}{3}} ( 2^{\frac{1}{3}} -1) \right)^{\frac{1}{3}} + 1  \right),\\
u_6:=v_{10} & = \frac{ \I \sqrt{3}}{3}\left(  \zeta_{3}  \left(  2^{\frac{1}{3}} ( 2^{\frac{1}{3}} -1) \right)^{\frac{1}{3}} + 1  \right),\\
u_7:=v^{'}_{00} & =\frac{\sqrt{3}}{3}\left(  \left(   \zeta_{12} 2^{\frac{1}{3}} ( 2^{\frac{1}{3}} +\zeta_{12}^{10}) \right)^{\frac{1}{3}} + \I  \right),\\
u_8:=v^{''}_{01} & =  \frac{\sqrt{3}}{3}\left(\left( -\zeta_{12} 2^{\frac{1}{3}} (\zeta_{12}^{10} 2^{\frac{1}{3}} +1) \right)^{\frac{1}{3}} +\I  \right),\\
u_9:=v^{''}_{11} & = \frac{\sqrt{3}}{3}\left(\zeta_3 \left(- \zeta_{12} 2^{\frac{1}{3}} (\zeta_{12}^{10} 2^{\frac{1}{3}} +1) \right)^{\frac{1}{3}} +\I   \right).\\
\end{aligned}
\end{equation}
Solving  the equations \ref{eq9-1} for these $u_i$'s and changing $t$ with $t=1$,   one obtains  five points $Q_i'$s for $i=5, \cdots, 9$ 
as given in the end of this section.

In order to find the tenth independent point, for $\ell=0, 1, 2$, we define
$w_\ell:=2^{\frac{1}{3}} 3^{\frac{2}{3}} (\zeta_3^\ell\, 2^{\frac{1}{3}} + \zeta_6)^{\frac{1}{3}}.$
Then, $\Phi_{3}(u)$  can be factored over 
$\Q(  \zeta_{12}, w_0^{\frac{1}{3}} )= 
\Q(  \zeta_{12},  6^{\frac{1}{9}}\,( (2^{\frac{1}{3}} -1))^{\frac{1}{9}}  )$ in the following form:
$$
		\Phi_{3} (u) =\displaystyle \prod_{\ell=0}^{2}  \prod_{i=0}^{2} 	\prod_{j=0}^{2} 	
\left[ u- \frac{1}{3}	\left(3^{\frac{1}{3}} \zeta_3^i  \left(   \frac{  \zeta_3^j w_\ell^{2}+2^{\frac{4}{3}} \zeta_3^{\ell} w_\ell+
	3 \zeta_3^{3-\ell-j} (2^{\frac{2}{3}}+2 \zeta_3^\ell)}{w_\ell}\right)^{\frac{1}{3} }-\frac{\zeta_3^\ell 2^{\frac{1}{3}} +1}{3}\right)  \right].\\
%
$$
We notice that the factor $\Phi_{4}(u)$ also can be factored over $\KK_9$ as follows:
$$
	\Phi_{4} (u) =\displaystyle \prod_{\ell=0}^{2}  \prod_{i=0}^{2} 	\prod_{j=0}^{2} 	
	\left[ u+ \frac{1}{3}	\left(3^{\frac{1}{3}}\zeta_3^i  \left(   \frac{  \zeta_3^j w_\ell^{2}+ 2^{\frac{4}{3}}  \zeta_3^{\ell} w_\ell+
		3 \zeta_3^{3-\ell-j} (2^{\frac{2}{3}}+2 \zeta_3^\ell)}{w_\ell}\right)^{\frac{1}{3} }-\frac{\zeta_3^\ell 2^{\frac{1}{3}}+1}{3}\right)  \right].\\
$$
We refer the reader to see \cite[deg27-1, deg27-2 ]{Shioda-Codes} for the above factorization.

Solving  the equations \ref{eq9-1} using $u_{10}$ defined as 
\begin{equation}
	\label{u10}
	u_{10}:=\frac{3^{\frac{1}{3}} }{3} \left(\frac{ w_0^{2}+   2^{\frac{4}{3}} w_0+ 3 (2^{\frac{2}{3}}+2)}{w_0}\right)^{\frac{1}{3}}-\frac{2^{\frac{1}{3}}+1}{3},
\end{equation}
and changing $t$ with $t+1$, one obtains  tenth point $ Q_{10}$. 

For the points $Q_j$'s as above we have,   $\left\langle Q_i, Q_i\right\rangle=3 $ for  $i=1,\cdots,8$, and  letting $x_{ij}=x(Q_i)-x(Q_j)$ and $ y_{ij}=y(Q_i)-y(Q_j)$ for
 $ 1\leq  i \neq  j \leq 10$,  the pairing  $\left\langle Q_i, Q_j\right\rangle  $   can be calculated by 
 \begin{align*}  
 	\left\langle Q_i, Q_j\right\rangle  & =2 - 
 	\left\lbrace  \deg \left( \gcd  (x_{ij} , y_{ij})  \right)    
 	+ \min\left\lbrace 3- \deg(x_{ij}), 4- \deg(y_{ij}))\right\rbrace \right\rbrace.
 \end{align*} 
 Hence,	the     Gram matrix of $Q_1, \cdots, Q_{10}$ is equal to
\begin{equation}\label{mat9}
	\M_9=
\begin{pmatrix} 
	3 & -\frac{3}{2} & -\frac{3}{2} & \frac{3}{2} & -\frac{3}{2} & -\frac{3}{2} & -\frac{3}{2} & -\frac{3}{2} & -\frac{3}{2} & \frac{1}{2} 
	\\ \noalign{\medskip}
	-\frac{3}{2} & 3 & 0 & 0 & 0 & 0 & 0 & 0 & 0 & -1 
	\\ \noalign{\medskip}
	-\frac{3}{2} & 0 & 3 & 0 & 1 & 1 & 1 & 1 & 1 & 0 
	\\ \noalign{\medskip}
	\frac{3}{2} & 0 & 0 & 3 & -1 & -1 & -1 & -1 & -1 & -1 
	\\ \noalign{\medskip}
	-\frac{3}{2} & 0 & 1 & -1 & 3 & 0 & 1 & 1 & 1 & 1 
	\\ \noalign{\medskip}
	-\frac{3}{2} & 0 & 1 & -1 & 0 & 3 & 1 & 1 & 1 & -1 
	\\ \noalign{\medskip}
	-\frac{3}{2} & 0 & 1 & -1 & 1 & 1 & 3 & 1 & 1 & 1 
	\\ \noalign{\medskip}
	-\frac{3}{2} & 0 & 1 & -1 & 1 & 1 & 1 & 3 & 0 & 0 
	\\ \noalign{\medskip}
	-\frac{3}{2} & 0 & 1 & -1 & 1 & 1 & 1 & 0 & 3 & -1 
	\\ \noalign{\medskip}
	\frac{1}{2} & -1 & 0 & -1 & 1 & -1 & 1 & 0 & -1 & 3 
\end{pmatrix}  
\end{equation}
which has determinant $3^5/4$ as desired. We cite 
\cite[points-9]{Shioda-Codes}
to see the list of points  and details of computation of the  above Gram matrix.

\section{Proof of Theorem~\ref{main-g}} 
\label{case10}
By the Table~\ref{tabusui}, the  lattice $\Ll_{10}$   is isomorphic to   $\Ll_{2}[5] \oplus \Ll_{5}[2]$, i.e., 
$$\Ee_{10}(\CC(t))=\Ee_2(\CC(t^5)) \oplus \Ee_5(\CC(t^2)) \iso \Ee_2(\KK_2 (t^5)) \oplus \Ee_5(\KK_5 (t^2))$$ and hence
$\Ee_{10}(\KK_{10}(t))$ is  of rank $r_{10}=10$ where  
$\KK_{10} =  \KK_{5},$
and a set of 
independent generators includes  two points 
$$Q_1=(-1, t^5), \ \  Q_2= \left(  -\zeta_3 \, \left( 1+\frac{ 4}{3} t^5\right) , 
( 2\zeta_3 +1 )\, t^5\, \left( 1 + \frac{ 8}{9} t^{10}\right) \right),$$  
coming from Theorem~\ref{main-a} by changing $t$ with $t^5$,    and 
 more eight points 
i.e.,
$$Q_{j+2}=\left(\frac{t^4 + a_j t^2 +b_j}{u_j^2}, \, \frac{t^6 +c_j t^4 + d_j t^2 +e_j}{u_j^3}\right), $$
for $\ j=1, \ldots, 8$, obtained by changing $t$ with $t^2$ from those given   in Theorem~\ref{main-d}.
 The Gram matrix $\M_{10}$ of these ten points is a diagonal blocked with $5 \M_2$ and $2 \M_5$ in its diagonal blocks, i.e.,

\begin{equation}\label{mat10}
\M_{10}=
	\left(\begin{array}{@{}c|c@{}}
		5 \M_2
		& \bigzero \\ \noalign{\medskip}
		\hline
		\bigzero &
		2 \M_5 
	\end{array}\right),
\end{equation}
and its determinant is $\det(\M_{10})= 5^2 2^8 /3$ as desired.
 Therefore, the proof of Theorem~\ref{main-g} is  completed.

\section{Proof of Theorem~\ref{main-h}} 
\label{case12}

In order to treat this case, we consider the elliptic $K3$ surface  $\Fc_{6}: y^2=x^3+t^{6}+1/t^{6}$. 
In \cite[Theorem~1.5]{Salami2022}, we have determined the splitting field $\KK$ and the generators of the Mordell--Weil lattice $\Fc_{6}(\CC(t))$. Indeed, we proved that
it is isomorphic to $\Fc_{6}(\KK(t)) $   and has rank  $r_6=16$, where
	$\KK$ is a number field with a defining minimal  polynomial $g_6(x)$ of degree 96
given in \cite[min-pols]{k3-codes}, and see also	\cite[minpols]{Shioda-Codes}.
 A set of independent generators includes the points
 $P'_j=(x_j(t), y_j(t) ) $ for $j=1,\cdots, 8$ with
\begin{align*}
	x'_j(t)& = \frac{a_{j,0} t^{4}+a_{j,1} t^{3}+a_{j,2} t^{2}+ a_{j,1} t +a_{j,0}}{t^2}, \\
	y'_j(t) &=\frac{b_{j,0} t^{6}+b_{j,1} t^{5}+b_{j,2} t^{4}+b_{j,3} t^{3}+b_{j,2} t^{2}+b_{j,1} t +b_{j,0}}{t^3},
	\end{align*}
%
	and   $P'_{j+8}= \phi_6(P'_j) (x'_{j+8}(t), y'_{j+8}(t))$  	for $j=1, \ldots, 8$,
 with coordinates
	$$\begin{aligned}[b]
		x'_{j+8}(t)& = \frac{a_{{j+8},4} t^{4}+a_{{j+8},3} t^{3}+a_{{j+8},2} t^{2}+ a_{{j+8},1} t +a_{{j+8},0}}{\zeta_{12}^2  t^2}, \\
		y'_{j+8}(t) &=\frac{b_{{j+8},6}t^{6}+b_{{j+8},5} t^{5}+b_{{j+8},4} t^{4}+b_{{j+8},3} t^{3}+b_{{j+8},2} t^{2}+b_{{j+8},1} t +b_{{j+8},0}}{\zeta_{12}^3  t^3},
	\end{aligned}$$
	are given in \cite[Points-6]{k3-codes}.
Since any  point  $\left(x'(t), y'(t) \right)$ on $\Fc_{6}: y^2=x^3+t^{6}+1/t^{6}$ can be transformed to  $(\zeta_{12}^{2}t^2 \, x'(t), \zeta_{12}^{3}t^3\, y'(t))$ on $\Ee_{12}: y^2=x^3+t^{12}+1$, 
 the  lattice $\Ee_{12}(\CC(t))$ is isomorphic to $\Fc_{12}(\CC(t))$, which implies  
$\KK_{12}=   \KK$. 
Let $P_j=\left(x(t), y(t) \right) =(\zeta_{12}^{2}t^2 \, x'(t), \zeta_{12}^{3}t^3\, y'(t))$ denotes the image of $P'_j=\left(x'(t), y'(t) \right)$ 
under this isomorphism. We provided  the 16 point on the elliptic surface $\Ee_{12}$ in 
\cite[Points-12]{Shioda-Codes}.

	Having  polynomial coordinates in $\KK_{12}[t]$, the  $P_j$'s have no intersection with zero sections of $\Ee_{12}$, we get that
$\langle P_j, P_j \rangle = 4$, and 
$\langle P_{j_1}, P_{j_2}\rangle = 2- (P_{j_1} \cdot P_{j_2})$, where
   for any $1\leq j_1 \neq j_2 \leq 16$, 
$(P_{j_1} \cdot P_{j_2})$ is the the intersection number given by
\begin{equation}
	( P_{j_1} \cdot P_{j_2})
	=\deg(\gcd (x_{j_1}- x_{j_2} , y_{j_1}- y_{j_2})) +
	\min \{4- \deg( x_{j_1}- x_{j_2}) , 6- \deg( y_{j_1}- y_{j_2})\}.
\end{equation}	
Thus,  we obtain the Gram matrix of    $P_1, \cdots, P_{16}$'s   as
{\small  
\begin{equation}\label{mat12}
	\setcounter{MaxMatrixCols}{20} 
	\M_{12}=
	\begin{pmatrix}
		4 & 2 & 0 & 0 & 0 & 2 & -1 & 2 & 0 & 0 & 0 & 0 & 0 & 0 & 0 & 0 \\ \noalign{\medskip}
		2 & 4 & 1 & 1 & 1 & 1 & -2 & 1 & 0 & 0 & 0 & 0 & 0 & 0 & 0 & 0 \\ \noalign{\medskip}
		0 & 1 & 4 & 0 & 0 & -2 & 0 & 0 & 0 & 0 & -1 & 1 & 0 & 0 & 0 & 0 \\ \noalign{\medskip}
		0 & 1 & 0 & 4 & -2 & 0 & -2 & 1 & 1 & 0 & 0 & 0 & 0 & 0 & 0 & 0 \\ \noalign{\medskip}
		0 & 1 & 0 & -2 & 4 & 0 & 1 & -2 & -1 & 0 & 0 & 0 & 0 & 0 & 0 & 0 \\ \noalign{\medskip}
		2 & 1 & -2 & 0 & 0 & 4 & 0 & 0 & 0 & 0 & 0 & -1 & 0 & 0 & 0 & 0 \\ \noalign{\medskip}
		-1 & -2 & 0 & -2 & 1 & 0 & 4 & -2 & -1 & 1 & 0 & 0 & 0 & 0 & 0 & 0 \\ \noalign{\medskip}
		2 & 1 & 0 & 1 & -2 & 0 & -2 & 4 & 0 & -1 & 0 & 0 & 0 & 0 & 0 & 0 \\ \noalign{\medskip}
		0 & 0 & 0 & 1 & -1 & 0 & -1 & 0 & 4 & -2 & 0 & 0 & 0 & 2 & 0 & 2 \\ \noalign{\medskip}
		0 & 0 & 0 & 0 & 1 & 0 & 1 & -1 & -2 & 4 & 0 & 0 & -2 & 0 & -2 & 0 \\ \noalign{\medskip}
		0 & 0 & -1 & 0 & 0 & 0 & 0 & 0 & 0 & 0 & 4 & -2 & 0 & 2 & 0 & -2 \\ \noalign{\medskip}
		0 & 0 & 1 & 0 & 0 & -1 & 0 & 0 & 0 & 0 & -2 & 4 & -2 & 0 & 2 & 0 \\ \noalign{\medskip}
		0 & 0 & 0 & 0 & 0 & 0 & 0 & 0 & 0 & -2 & 0 & -2 & 4 & -2 & 0 & 0 \\ \noalign{\medskip}
		0 & 0 & 0 & 0 & 0 & 0 & 0 & 0 & 2 & 0 & 2 & 0 & -2 & 4 & 0 & 0 \\ \noalign{\medskip}
		0 & 0 & 0 & 0 & 0 & 0 & 0 & 0 & 0 & -2 & 0 & 2 & 0 & 0 & 4 & -2 \\ \noalign{\medskip}
		0 & 0 & 0 & 0 & 0 & 0 & 0 & 0 & 2 & 0 & -2 & 0 & 0 & 0 & -2 & 4
	\end{pmatrix}
\end{equation}}

We notice that the     determinant of $\M_{12}$ is  $2^4 3^4$ as desired.
 Therefore, the   points
$P_j$'s for $j=1, \ldots, 16$ form a set of independent generators of $\Ee_{12} $	over $\KK_{12}(t)$.


\bibliographystyle{amsplain} 

\bibliography{SBIB}{} 

\end{document}